\documentclass[11pt]{amsart}

\usepackage{amssymb}
\usepackage{verbatim}

\begin{document}

\newtheorem{theorem}{Theorem}    
\newtheorem{proposition}[theorem]{Proposition}
\newtheorem{conjecture}[theorem]{Conjecture}
\def\theconjecture{\unskip}
\newtheorem{corollary}[theorem]{Corollary}
\newtheorem{lemma}[theorem]{Lemma}
\newtheorem{observation}[theorem]{Observation}
\theoremstyle{definition}
\newtheorem{definition}{Definition}
\newtheorem{notation}[definition]{Notation}
\newtheorem{remark}[definition]{Remark}
\newtheorem{question}[definition]{Question}
\newtheorem{example}[definition]{Example}
\newtheorem{problem}[definition]{Problem}
\newtheorem{exercise}[definition]{Exercise}

\numberwithin{theorem}{section}
\numberwithin{definition}{section}
\numberwithin{equation}{section}

\def\earrow{{\mathbf e}}
\def\uarrow{{\mathbf u}}

\def\reals{{\mathbb R}}
\def\torus{{\mathbb T}}
\def\heis{{\mathbb H}}
\def\integers{{\mathbb Z}}
\def\complex{{\mathbb C}\/}
\def\naturals{{\mathbb N}\/}
\def\distance{\operatorname{distance}\,}
\def\dist{\operatorname{dist}\,}
\def\Span{\operatorname{span}\,}
\def\degree{\operatorname{degree}\,}
\def\kernel{\operatorname{kernel}\,}
\def\dim{\operatorname{dim}\,}
\def\codim{\operatorname{codim}}
\def\trace{\operatorname{trace\,}}
\def\Span{\operatorname{span}\,}
\def\ZZ{ {\mathbb Z} }
\def\p{\partial}
\def\rp{{ ^{-1} }}
\def\Re{\operatorname{Re\,} }
\def\Im{\operatorname{Im\,} }
\def\ov{\overline}
\def\eps{\varepsilon}
\def\lt{L^2}
\def\diver{\operatorname{div}}
\def\curl{\operatorname{curl}}
\def\etta{\eta}
\newcommand{\norm}[1]{ \|  #1 \|}
\def\Span{\operatorname{span}}

\newcommand{\Norm}[1]{ \left\|  #1 \right\| }
\newcommand{\set}[1]{ \left\{ #1 \right\} }

\def\scriptf{{\mathcal F}}
\def\scriptg{{\mathcal G}}
\def\scriptm{{\mathcal M}}
\def\scriptb{{\mathcal B}}
\def\scriptc{{\mathcal C}}
\def\scriptt{{\mathcal T}}
\def\scripti{{\mathcal I}}
\def\scripte{{\mathcal E}}
\def\scriptv{{\mathcal V}}
\def\frakv{{\mathfrak V}}

\author{Michael Christ}
\address{
        Michael Christ\\
        Department of Mathematics\\
        University of California \\
        Berkeley, CA 94720-3840, USA}
\email{mchrist@math.berkeley.edu}
\thanks{The author was supported in part by NSF grant DMS-040126}

\date{September 14, 2005. Revised June 3, 2011.}

\title
{Quasiextremals for a Radon-like Transform}

\maketitle


\section{Introduction}
\subsection{An operator}
The object of our investigation is
the linear operator
$T$, mapping functions defined on $\reals^d$
to functions defined on $\reals^d$, defined by
\begin{equation}
\label{parabolaoperator}
Tf(x) = \int_{\reals^{d-1}}f(x'-t,x_d-|t|^2)\,dt
\end{equation}
where $x=(x',x_d)\in\reals^{d-1}\times\reals^1$.
$T$ is one of the most basic examples of a quite broad class
of generalized Radon transforms, and more generally, of Fourier
integral operators. 
These generalized Radon transforms take the form 
\begin{equation} \label{generalRadon}
Tf(x) = \int_{\scriptm_x}f(y)\,d\sigma_x(y),
\end{equation}
where for each $x$ in some ambient manifold,
each set $\scriptm_x$ is a smooth submanifold of a second ambient manifold,
$\scriptm_x$ varies smoothly with $x$,
and $\sigma_x$ is a smooth multiple of the induced surface measure
on $\scriptm_x$. A transversality hypothesis is also imposed,
guaranteeing that the transpose of $T$ is 
a generalized Radon transform in the same sense.

As is well known,
the particular operator
$T$ defined by \eqref{parabolaoperator}
maps $L^{(d+1)/d}(\reals^d)$ boundedly to $L^{d+1}(\reals^d)$,
but does not map
$L^p$ boundedly to $L^q$ for any other exponents $p,q$.
The localized operator
\begin{equation}
T_0f(x) = \int_{|t|\le 1} f(x'-t,x_d-|t|^2)\,dt,
\end{equation}
with $x$ restricted to a fixed bounded subset of $\reals^d$,
does obey a wider class of $L^p\mapsto L^q$ inequalities,
but all of them are consequences of this most
basic inequality by interpolation with trivial estimates
and H\"older's inequality.

The 
$L^{(d+1)/d}(\reals^d)\mapsto L^{d+1}(\reals^d)$ inequality,
in much greater generality, was originally
established by arguments relying on $L^2$ smoothing properties,
which in turn were established by Fourier transform
or Fourier integral operator theory.
In this paper we use combinatorial methods to establish
refinements of this inequality.

These refinements are of three types.
\newline
(i) A rough characterization is given of
quasiextremals, by which we mean functions $f$ for which 
$\norm{Tf}_{d+1}/\norm{f}_{(d+1)/d}$ 
is at least a constant multiple of the supremum of
this ratio over all functions. This constant can
be arbitrarily small.
\newline
(ii) 
Theorem~\ref{thm:reformulation} 
asserts that if $f$ is sparsely distributed
in a certain precise sense, then $\norm{Tf}_{d+1}/\norm{f}_{(d+1)/d}$ is small. 
\newline
(iii) It is shown that $T$ maps $L^{(d+1)/d}$
to the Lorentz space $L^{d+1,r}$ for any $r > \frac{d+1}{d}$;
these spaces are strictly smaller than $L^{d+1}$ when $r<d+1$.
The range of $r$ is optimal, except perhaps for
the endpoint case $r=\frac{d+1}{d}$, which remains open.
Underlying this extension is a general functional analytic
framework for passing from restricted weak type inequalities
to strong type, and more general Lorentz type, inequalities.
For such an extrapolation specific additional information, which here
takes the form of a certain multilinear inequality, is
also needed; see Lemma~\ref{lemma:trulymultilinear}.
This formalism has already been exploited by Stovall \cite{betsy}
to prove strong type inequalities for a different class of
Radon-like transforms, for which only restricted weak type
estimates had previously been known.
It has also been applied by Dendrinos, Laghi, and Wright \cite{DLW}
to another related class of transforms.
This formalism does not rely on the characterization of quasiextremals;
it is less specific and hence more flexible.

\medskip
The particular operator \eqref{parabolaoperator} 
is distinguished from others of the 
form \eqref{generalRadon} by the presence of a 
group of associated symmetries of quite high dimension. 
These symmetries are central to
our discussion, and dictate the form of the results.

More general operators of the same general class enjoy
fewer symmetries, and the most straightforward extensions
of the main results of this paper to those generalizations
are false. See for example Stovall's characterization \cite{betsyspheres} 
of quasiextremals
for the operator defined by convolution with surface measure on 
a sphere in $\reals^d$. 
The techniques developed here are nonetheless the basis of 
further work \cite{betsyspheres},\cite{betsygeneric} which, with
further ideas, establishes the correct extensions. 

It is natural to ask why the measure $dt$ is employed
in the definition \eqref{parabolaoperator} of $T$, rather than surface measure
on the paraboloid pulled back to $\reals^{d-1}$.
A partial answer is that $dt$ possesses a dilation symmetry
which surface measure lacks. A fuller answer may be found
in the discussion of affine surface measure in \cite{christextremal}.

\subsection{Restricted weak type inequality}
The slightly weaker restricted weak type formulation 
of the $L^{(d+1)/d}\mapsto L^{d+1}$ inequality
says that for any two measurable sets,
\begin{equation} \label{restrictedweaktype}
\langle T(\chi_{E^\star}),\,\chi_E\rangle \lesssim |E|^{d/(d+1)}|E^\star|^{d/(d+1)},
\end{equation}
where $\chi_E$ denotes the characteristic function of $E$.
Combinatorial proofs of \eqref{restrictedweaktype}
have been given in \cite{schlag} and \cite{edinburgh}.

\eqref{restrictedweaktype} has a more geometric interpretation
than does the $L^p$ norm inequality. 
Denote by $\scripti\subset\reals^d\times\reals^d$ the incidence manifold
\begin{equation}
\scripti = \set{(x,y)\in\reals^{d+d}: y_d = x_d - |y'-x'|^2}
\end{equation}
where $x=(x',x_d)$ and $y=(y',y_d)$.
Let $\pi,\pi^{\star}:\scripti\to\reals^d$ be the projections
\begin{equation}
\pi(x,x_\star)=x \text{ and } \pi^{\star}(x,x_\star)=x_\star.
\end{equation}
Then 
\[
\langle T(\chi_{E^\star}),\,\chi_E\rangle 
= c|\scripti\cap(E\times E^\star)|,
\]
where $|\cdot|$ denotes Lebesgue measure on $\scripti$, 
and thus 
$\langle T(\chi_{E^\star}),\,\chi_E\rangle$
represents the continuum number of incidences between $E,E^\star$.

The restricted weak type inequality \eqref{restrictedweaktype}
is sharp not only in the sense that neither exponent $\frac{d}{d+1}$ can be
increased without decreasing the other, but moreover,
for any $t,t_\star>0$
there exist sets $E,E^\star$ satisfying $|E|=t$ and 
$|E^\star|=t_\star$ with
$\langle T(\chi_{E^\star}),\,\chi_E\rangle \ge c|E|^{d/(d+1)}|E^\star|^{d/(d+1)}$
where $c>0$ is independent of $t,t_\star$.
Our refinement will quantify the principle that this inequality
can nonetheless be improved for typical sets.

\subsection{Definition of quasiextremals}
To formulate refinements requires a definition.
\begin{definition}  \label{defn:QEpairs}
Let $\eps\in\reals^+$.
An ordered pair $(E,E^\star)$ of Lebesgue measurable subsets of
$\reals^d$ is said to be $\eps$-quasiextremal 
for the inequality \eqref{restrictedweaktype}
if $0<|E|,|E^\star|<\infty$ and
\begin{equation}
\langle T(\chi_{E^\star}),\,\chi_E\rangle \ge\eps |E|^{d/(d+1)}|E^\star|^{d/(d+1)}.
\end{equation}
\end{definition}
\noindent
We will say simply that $(E,E^\star)$ is $\eps$-quasiextremal.

The first goal of this note
is to identify, in a natural sense, all $\eps$-quasiextremal pairs, 
thereby refining the norm inequalities already known.
This is rather different from the general problem of
identifying all {\em exact} extremals and finding the optimal constants
in the strong type and restricted weak type inequalities, 
concerning which we have nothing to contribute.
Here we are interested in pairs that are extremal merely up to the factor $\eps$.  
There are several natural asymptotic regimes for $\eps$.
The simplest has $\eps$ bounded below, while in the second,
$\eps$ tends to zero; both of these are addressed by our results.
In this paper we obtain no additional information when
$\eps$ approaches, or equals, the optimal constant in the
inequality, but those situations are the topic of a subsequent
work \cite{christextremal}. 

An alternative formulation of quasiextremality is more natural
for more general operators.
For any $t,t^\star>0$ define
\begin{equation*}
\Lambda(t,t_\star)
=
\sup_{|E|=t,\ |E^\star|=t_\star}
t_{\phantom{\star}}^{-d/(d+1)}
t_\star^{-d/(d+1)}
\langle T(\chi_{E^\star}),\,\chi_E\rangle.
\end{equation*}
One could then define an $\eps$-quasiextremal pair by the inequality
\begin{equation*}
\langle T(\chi_{E^\star}),\,\chi_E\rangle
\ge\eps \Lambda(|E|,|E^\star|).
\end{equation*}
For the particular operator \eqref{parabolaoperator},
it turns out that $\Lambda(t,t_\star)\sim t^{d/(d+1)}t_*^{d/(d+1)}$
for all $t,t_\star$.
For the localized operator $T_0$, however, the relationship between these
two alternative notions of quasiextremality is more complicated. 
See the discussion following Theorem~\ref{thm:Lambda} below.

\subsection{A family of quasiextremals}
We first describe a family of quasiextremals,
that is, $\eps$-quasiextremals with $\eps$ bounded below
by a fixed positive constant.

\begin{definition}
\label{balldefinition}
For any point $\bar z=(\bar x,\bar x_\star)\in \scripti$, any $\rho>0$, 
any orthonormal basis $\earrow=\{e_1,\cdots,e_{d-1}\}$ for $\reals^{d-1}$, 
and any $r,r^{\star}\in (\reals^+)^{d-1}$ satisfying
\begin{equation}
\label{dualellipses}
r_jr^{\star}_j=\rho\ \forall\, 1\le j\le d-1
\end{equation}
$\scriptb(\bar z,\earrow,r,r^{\star})$ denotes the set of all
$z=(x,x_\star)\in \scripti$ satisfying all of  
\begin{gather}
\label{firstfirstcoords}
|\langle x'-\bar x',e_j\rangle|<r_j\ \forall\, j,
\\
\label{firstlastcoord}
\big| x_d-(\bar x_\star)_d -|x'-\bar x_\star'|^2\big| <\rho,
\\
\label{secondfirstcoords}
|\langle x_\star'-\bar x_\star',e_j\rangle|<r^{\star}_j\ \forall\, j,
\\
\label{secondlastcoord}
\big| (x_\star)_d - \bar x_d + |x_\star'-\bar x'|^2 \big| <\rho.
\end{gather}
\end{definition}
$\scriptb$ is by definition the intersection of $\scripti$
with a certain Cartesian product $E\times E^\star$, 
whence $\pi(\scriptb)\subset E$ and $\pi^\star(\scriptb)\subset E^\star$.
In fact, $\pi(\scriptb),\pi^\star(\scriptb)$ are essentially equal
to $E,E^\star$; see the proof of Proposition~\ref{prop:trulyQE}
in \S\ref{section:trulyQE}.

Our canonical quasiextremal pairs will be all ordered pairs $(E,E^\star)= 
(\pi\scriptb,\pi^{\star}\scriptb)$, where 
$\scriptb= \scriptb(\bar z,\earrow,r,r^{\star})$ 
is any of the balls defined above.
\begin{proposition}
\label{prop:trulyQE}
There exists $c_0>0$ such that uniformly
for all sets $\scriptb$
described in Definition~\ref{balldefinition},
the pair of sets $(E,E^\star)=(\pi(\scriptb),\pi^\star(\scriptb))$ 
is $c_0$-quasiextremal
for the inequality \eqref{restrictedweaktype}, that is,
$\langle T(\chi_{\pi^\star(\scriptb)}),\,\chi_{\pi(\scriptb)}\rangle 
\ge c_0
|\pi(\scriptb)|^{d/(d+1)}|\pi^\star(\scriptb)|^{d/(d+1)}$.
\end{proposition}
\noindent  The straightforward verification of this
claim is postponed to \S\ref{section:trulyQE}.

These sets are numerous; $\scriptb$ depends on $2d-1 + \tfrac12 d(d+1)$ free parameters.
All of them are derived from a single example via the application of
geometric symmetries, discussed below.

We will call these sets 
``balls'' in recognition of the partial analogy with
balls introduced in connection
with various problems in harmonic analysis, partial differential
equations, and complex analysis in several variables; see for instance 
\cite{nsw},\cite{brunanagelwainger},\cite{feffermanphong},\cite{parme},\cite{christspeculation},\cite{mcneal},\cite{mcnealstein}.
However, whereas those other types of balls are associated 
to certain metrics in the sense of point-set topology,
the sets $\scriptb$
do {\em not} seem to be naturally associated with metrics. 
It seems to be an interesting question what the analogous geometric structures are,
if any, for other Radon-like transforms defined by integration over submanifolds
of dimension strictly greater than one. 
We maintain that the sets defined by
Definition~\ref{balldefinition}
are natural analogues, for our particular operator $T$, of the balls associated
to Radon-like transforms defined by integration over one-dimensional manifolds
\cite{taowright}.
The family of sets $\scriptb$ is studied in more detail in \cite{christextremal}.

\subsection{Main result}

If there is some $\scriptb$ such that
$E$ is the union of $\pi(\scriptb)$ with an arbitrary set having
measure $\eps^{-1}|\pi(\scriptb)|$, and likewise $E^\star$ is the union
of $\pi^\star(\scriptb)$ with an arbitrary set of measure
$\eps^{-1}|\pi^\star(\scriptb)|$,
then $(E,E^\star)$ is $c\eps$-quasiextremal.  Thus 
$\eps$-quasiextremality for small $\eps$ 
cannot impose structure on more than
small portions of $E,E^\star$.

Write $\scriptt(E,E^{\star}) = \langle T(\chi_{E^{\star}}),\,\chi_E\rangle$.
\begin{theorem} \label{maintheorem}
Let $d\ge 2$.
There exist $C,A<\infty$ with the following property.
For any $\eps>0$ and
any measurable sets
$E,E^{\star}\subset\reals^d$ of positive Lebesgue measures satisfying
$\scriptt(E,E^{\star})
\ge \eps|E|^{d/(d+1)} |E^{\star}|^{d/(d+1)}$,
there exists a set $\scriptb\subset \scripti$, 
of the type described in Definition~\ref{balldefinition},
such that the associated pair
$(B,B^{\star}) = (\pi(\scriptb),\pi^{\star}(\scriptb))$ satisfies
\begin{gather}
\scriptt(E\cap B,E^{\star}\cap B^{\star}) \ge C^{-1}\eps^A \scriptt(E,E^{\star})
\label{Abound}
\\
\intertext{and}
|B|\le |E| \text{ and } |B^{\star}|\le |E^{\star}|.
\label{Bbound}
\end{gather}
\end{theorem}

The proof of Theorem~\ref{maintheorem}
yields a slightly stronger conclusion:
there exists a pair $(B,B^\star)$ satisfying
\begin{gather}
\scriptt(E\cap B,E^{\star}\cap B^{\star}) \ge C^{-1}\scriptt(E,E^{\star})
\label{eq:strongerconclusionA}
\\
\intertext{and}
|B|\le C\eps^{-A}|E| \text{ and } |B^{\star}|\le C\eps^{-A}|E^{\star}|.
\label{eq:strongerconclusionB}
\end{gather}
This implies \eqref{Abound},\eqref{Bbound} by a simple covering argument,
Lemma~\ref{lemma:covering}.

A result related in spirit, though for a different operator, is an inequality of
Moyua, Vargas, and Vega \cite{MVV}. A fundamental difference here is the more complicated
geometry; the set of all sets $B$ of the type described above having a specified measure,
does not enjoy simple covering properties like those of the set of all Euclidean balls
of specified measure.

Theorem~\ref{maintheorem} does not {\em characterize} quasiextremal {\em pairs},
even disregarding the ambiguity inherent in the exponent $A$.
There exists a constant $\delta>0$ such that
for any $\scriptb$, there exist sets $E\subset B=\pi(\scriptb)$,
$E^\star\subset B^\star=\pi^\star(\scriptb)$
satisfying $|E|\ge\delta |B|$ and $|E^\star|\ge\delta |B^\star|$,
yet $\langle T(\chi_{E^\star}),\,\chi_E\rangle=0$.
The analysis does give some further information about quasiextremal pairs, but
we do not know how to formulate it in a useful way.
However, more can be said about single quasiextremal sets and functions.
See Theorem~\ref{thm:single} below.

A further refinement would be to obtain the optimal value for the exponents
$A$ in \eqref{Abound} or
\eqref{eq:strongerconclusionB},
or $\delta$ in \eqref{reformulatedbound}, below.
Some concrete numbers could be extracted from the proof, but we have investigated 
neither their value nor their optimality.

Theorem~\ref{maintheorem} can be equivalently reformulated 
as a refinement of the restricted weak type bound.
\begin{theorem} \label{thm:reformulation}
There exist $C<\infty$ and $\delta>0$ such that for any measurable sets
$E,E^{\star}\subset\reals^d$ of positive Lebesgue measures, 
\begin{equation} \label{reformulatedbound}
\langle T(\chi_{E^{\star}}),\,\chi_E\rangle
\le C|E|^{d/(d+1)}
|E^\star|^{d/(d+1)}
\cdot
\sup_{\scriptb}
\Big(
\frac{|E\cap \pi(\scriptb)|}{|E|}
\cdot
\frac{|E^\star\cap \pi^\star(\scriptb)|}{|E^\star|}
\Big)^\delta
\end{equation}
where the supremum is taken over all $\scriptb$
described in Definition~\ref{balldefinition}
satisfying $|\pi(\scriptb)|\le |E|$ 
and $|\pi^\star(\scriptb)|\le |E^\star|$.
\end{theorem}

\subsection{A local analogue}
The situation for
the localized operator $T_0$ 
is more complicated to describe, though not more subtle.
Recall that $T_0$ maps
$L^p$ to $L^q$ if and only if
$(p^{-1},q^{-1})$ belongs to the convex hull
of $(0,0)$, $(1,1)$, $(0,1)$, and $(\frac{d}{d+1},\frac1{d+1})$.
These inequalities follow from the
$L^{(d+1)/d}\to L^{d+1}$ inequality via interpolation with
the trivial $L^1\mapsto L^1$ and $L^\infty\mapsto L^\infty$ bounds.

Define
\begin{equation} \label{eq:Lambdadefn}
\Lambda_0(t,t_\star) = 
\sup_{|E|=t}
\,\,
\sup_{|E^\star|=t_\star}
\,\,
\langle T_0(\chi_{E^\star}),\chi_E\rangle.
\end{equation}
$(E,E^\star)$ is said to be $\eps$-quasiextremal
with respect to the functional $\Lambda_0$
if $\langle T_0(\chi_{E^\star}),\chi_{E}\rangle 
\ge \eps\Lambda(|E|,|E^\star|)$.

Since $T_0$ preserves both $L^1$ and $L^\infty$, there is the bound
\begin{equation*}
\langle T_0(\chi_{E^\star}),\chi_{E}\rangle
\le C\min(|E|,|E^\star|,|E|^{d/(d+1)}|E^\star|^{d/(d+1)}),
\end{equation*}
so $\Lambda_0(t,t_\star)
\le C\min\big(t,t_\star,t^{d/(d+1)}t_\star^{d/(d+1)}\big)$.
Simple examples demonstrate that there are no stronger power law bounds; 
\begin{equation*}
\Lambda_0(t,t_\star)\sim \min(t,t_\star,t^{d/(d+1)}t_\star^{d/(d+1)})
\end{equation*}
uniformly for all $t,t_\star$.
Note that $\Lambda_0(t,t_\star)\sim t^{d/(d+1)}t_\star^{d/(d+1)}$
if and only if
$|t|\gtrsim |t_\star|^d$ and $|t_\star|\gtrsim |t|^d$;
otherwise the upper bound $\min(|t|,|t_\star|)$ is more restrictive.

\begin{theorem} \label{thm:Lambda}
Let $c>0$ be arbitrary.
Suppose that 
$\langle T_0(\chi_{E^\star}),\chi_E\rangle
\ge \eps\Lambda_0(|E|,|E^\star|)$,
and moreover that 
\begin{equation} \label{eq:weaklycomparable}
|E|\ge c|E^\star|^d
\text{ and }
|E^\star|\ge c|E|^d.
\end{equation}
Then there exist $\scriptb,B,B^\star$
there exists a set $\scriptb\subset \scripti$,
of the type described in Definition~\ref{balldefinition},
such that the associated pair
$(B,B^{\star}) = (\pi(\scriptb),\pi^{\star}(\scriptb))$ satisfies
\begin{equation*}
\big\langle T_0(\chi_{E^{\star}\cap B^{\star}}),
\chi_{E\cap B}\big\rangle
\ge C^{-1}\eps^A 
\big\langle T_0(\chi_{E^{\star}}),
\chi_{E}\big\rangle
\end{equation*}
with $|B|\le C|E|$  and  $|B^{\star}|\le C|E^{\star}|$.
\end{theorem}
\noindent
This is a direct consequence of Theorem~\ref{maintheorem},
since \eqref{eq:weaklycomparable} implies that
$\Lambda_0(|E|,|E^\star|)$ is comparable to
$|E|^{d/(d+1)}|E^\star|^{(d/(d+1)}$.

No reasonable analogue of the conclusion holds without
the supplementary hypothesis \eqref{eq:weaklycomparable}.
Perhaps the simplest example illustrating this is
where $E^\star$ is the unit ball $B(0,1)$, and $E$ is
an arbitrary subset of $B(0,\tfrac12)$ of small measure.
Then $\langle T_0(\chi_{E^\star}),\chi_E\rangle
\sim |E|\sim\Lambda_0(|E|,|E^\star|)$,
uniformly over all $E\subset B(0,\tfrac12)$.

To construct a second class of trivial examples,
consider any positive integer $N$
and any subset $\{z_j: 1\le j\le N\}$ of $\reals^d$
of cardinality $N$. 
Let $F\subset\reals^d$ be the union of the paraboloids
$P_j=\{z_j-(t,|t|2): t\in\reals^{d-1} \text{ and } |t|<1\}$.
Let $E^\star_\delta$ be the set of all points within distance
$2\delta$ of $\cup_{j=1}^N P_j$,
and let $E_\delta=\cup_{j=1}^N B(z_j,\delta)$ be the union of the $N$
$\delta$-balls centered at the points $z_j$.
If $\delta\in(0,1]$ is chosen to be sufficiently small,
depending on $\{z_j\}$, then $|E_\delta^\star|\sim N\delta$,
while $|E_\delta|\sim N\delta^d$, uniformly in $N,\delta$
provided that $\delta$ is sufficiently small.
Thus $|E_\delta|\lesssim |E^\star_\delta|^d$,
whence $\Lambda_0(|E_\delta|,|E_\delta^\star|)\sim |E_\delta|$.
Moreover
$|E_\delta|\ll |E^\star_\delta|^d$ as $N\to\infty$.
Clearly $T(\chi_{E^\star_\delta})\gtrsim 1$
at every point of $E_\delta$, uniformly in all parameters,
whence 
$\langle T_0(\chi_{E^\star}),\chi_E\rangle
\gtrsim 
\Lambda_0(|E|,|E^\star|)$.

\subsection{Three extensions}
Theorem~\ref{maintheorem} 
has an extension to general functions.
We say that a pair of functions $(f,f^\star)$ is $\eps$-quasiextremal
if both $f,f^\star$ have finite $L^{(d+1)/d}$ norms
and 
\begin{equation} \label{eq:QEforfunctions}
|\langle Tf^\star,f\rangle|\ge \eps
\|f\|_{(d+1)/d} \|f^\star\|_{(d+1)/d}.
\end{equation}

\begin{theorem} \label{thm:functions}
There exist $c,A\in\reals^+$ such that for any $\eps>0$,
for any pair of nonnegative functions $(f,f^\star)$
which is $\eps$-quasiextremal 
in the sense of inequality \eqref{eq:QEforfunctions},
there exist sets $E,E^\star$, positive scalars $t,t^\star$,
and a ball $\scriptb$ of the type described in
Definition~\ref{balldefinition},
such that
\begin{gather}
t\chi_E\le f \text{ and } 
t^\star\chi_{E^\star}\le f^\star 
\\
\langle T(t^\star\chi_{E^\star\cap B^\star}),\,t\chi_{E\cap B}\rangle
\ge c \eps^A
\langle T(f^\star),\,f\rangle
\\
|B|\le|E| \text{ and } |B^\star|\le|E^\star|,
\end{gather}
where $(B,B^\star) =(\pi(\scriptb), \pi^\star(\scriptb))$.
\end{theorem}

The proof leads naturally to Lorentz space inequalities.
Denote by $L^{p,r}$ the usual Lorentz spaces \cite{steinweiss}. 
Any measurable function
function $f$ is expressed uniquely, modulo null sets, 
as $f(x) = \sum_{k\in\integers} 2^k f_k(x)$
where $\chi_{E_k}(x)\le |f_k(x)|< 2\chi_{E_k}(x)$ and
the sets $E_k$ are pairwise disjoint.
Then the $L^{p,r}$ norm of $f$ is comparable to
$(\sum_{k\in\integers} (2^k|E_k|^{1/p})^r)^{1/r}$;
$L^{p,r}$ is the set of all functions having finite norms.
$L^{p,r}$ embeds properly in $L^p$ whenever $r<p$.
\begin{theorem} \label{thm:Lorentz}
$T$ maps $L^{(d+1)/d}$ boundedly to the Lorentz space
$L^{d+1,r}$ for all $r>(d+1)/d$.
\end{theorem}
This statement is nearly optimal; no such bound can hold for $r<(d+1)/d$.
However, our method leaves open the endpoint $r=(d+1)/d$.
The proof of Theorem~\ref{thm:Lorentz} introduces general ideas
which should be useful in other problems.
A novel feature of the argument 
is its reliance on a trilinear variant of the analysis. 


In the case $d=2$, Lorentz space bounds,
including the endpoint $r=(d+1)/d$ not reached by our method,
are established in greater generality in \cite{bakoberlinseeger}.
It has been noted \cite{christradonextremals}
that certain changes of variables intertwine $L^{p,r}\to L^{d+1}$ inequalities for $T$  
with corresponding inequalities for the Radon transform. It has been shown \cite{christkplane}
that the Radon transform maps $L^{(d+1)/d,d+1}$ boundedly to $L^{d+1}$.
Reversing the changes of variables establishes the same mapping property for $T$. 

However, the reasoning in \cite{christkplane} relies on 
the exponent $d+1$ being an integer, a fact which plays no role in the present paper.
Indeed, Stovall \cite{betsy} has combined the method introduced 
here in the proof of Theorem~\ref{thm:Lorentz}
with an extension of the analysis in \cite{ccc}
to establish strong type endpoint bounds for the
Radon-like transforms defined by convolution with
smooth measures on the curves $(t,t^2,t^3,\cdots t^d)$
in $\reals^d$, for which restricted weak type bounds
were established in \cite{ccc}. In that situation, the corresponding
exponents are not integers, so the multilinear approach does not seem to be applicable.

These results lead directly to information about
individual sets or functions --- as opposed to pairs of
sets or functions --- which are quasiextremal in the natural sense. Here
are some of the possible formulations.
In the following theorem, $\scriptb$ always denotes a set
of the type introduced in Definition~\ref{balldefinition}.

\begin{theorem} \label{thm:single}
(i)
If $E$ is a measurable set such that 
$\norm{T(\chi_E)}_{L^{d+1,\infty}} \ge\eps|E|^{d/(d+1)}$
then
there exists $\scriptb$ 
which satisfies
$|\pi^\star(\scriptb)\cap E|\ge c\eps^C|E|$.
Conversely, for any set $\scriptb$ described in
Definition~\ref{balldefinition}, for any set $E\subset\pi^\star(\scriptb)$,
$\norm{T(\chi_E)}_{L^{d+1,\infty}}\ge c(|E|/|\pi^\star(\scriptb)|)^C |E|^{d/(d+1)}$. 

(ii) If $f$ is a nonnegative measurable function satisfying
$\norm{T(f)}_{L^{d+1}} \ge\eps \norm{f}_{L^{(d+1)/d}}$
then there exist a scalar $r\in\reals^+$, a measurable
set $E$, and a set $\scriptb$ of the type introduced in \eqref{balldefinition}
such that 
\begin{align*}
&r\chi_E\le f,
\\
&\norm{r\chi_E}_{L^{(d+1)/d}}\ge c\eps^C\norm{f}_{L^{d+1)/d}},
\\
&|\pi^\star(\scriptb)|\le |E|,
\\
&|\pi^\star(\scriptb)\cap E|\ge c\eps^C|E|.
\end{align*}

(iii) There exist $c,C \in\reals^+$
such that for any $\eps>0$,
if $f\in L^{(d+1)/d}$ is any complex-valued function satisfying
$\norm{T(f)}_{L^{d+1}} \ge\eps \norm{f}_{L^{(d+1)/d}}$
then 
there exist $r\in\reals^+$ and an $C\eps^{-C}$--bump function $\varphi$
such that 
\begin{equation} \label{complexvaluedconverse}
\norm{f-r\varphi}_{L^{(d+1)/d}}\le (1-c\eps^C)\norm{f}_{(d+1)/d}.
\end{equation}
\end{theorem}

$L^{p,r}$ again denotes a Lorentz space, with the standard notation.
The notion of an $\eps$--bump function requires definition.
Let $Q_0$ be the open cube in $\reals^d$ consisting of
of all points $(x_1,\cdots,x_d)$ satisfying $|x_j|<1$ for all $1\le j\le d$.
To our set $\scriptb=\scriptb(\bar z,\earrow,r,r^\star)$ is associated a canonical one-to-one
correspondence $\Phi_\scriptb:\pi^\star(\scriptb)\to Q_0$.
Then an $\eps$--bump function associated to $\pi^\star(\scriptb)$ is any function 
of the form $\varphi= \psi\circ \Phi_\scriptb$
where $\psi\in C^1$ is supported in $Q_0$
and satisfies
$\norm{\psi}_{C^1}\le \eps^{-1}$
and
$\norm{\psi}_{C^0}\ge 1$.
An $\eps$--bump function is then
any such function associated to $\pi^\star(\scriptb)$ 
for some $\scriptb$.

In part (ii), there is of course a converse, by part (i).
Likewise in (iii), $r\varphi$ 
is a $c\eps^C$--quasiextremal.
For the condition 
$\norm{f-r\varphi}_{L^{(d+1)/d}}\le (1-c\eps^C)\norm{f}_{(d+1)/d}$
imposes upper and strictly positive lower bounds on the
coefficient $r$. 
$T$ is a unitary convolution operator on $L^2(\reals^d)$, as one sees by computing
the associated Fourier multiplier. 
Therefore $\norm{T\psi}_{L^2}$ satisfies a strictly positive lower bound.
Since $\psi$ has bounded $C^1$
norm and is supported in $Q_0$, $T\psi$ is also {\it a priori}
bounded above in $C^1$. 
An elementary argument shows that $|T\psi(x)|\le C_\eps |x|^{-1/2}$,
and $T\psi$ is supported in a tubular neighborhood of fixed width
of a paraboloid. 
These facts together imply an {\it a priori}
lower bound on $\norm{T\psi}_{L^{d+1}}$.
This is only a partial converse, to be sure; 
\eqref{complexvaluedconverse} does not directly imply that $f$ is a quasiextremal.

\medskip
The symbols $c,C$ are sometimes used to denote 
positive finite constants whose values may change from one
occurrence to the next. Typically $c$ will be assumed to be sufficiently
small, while $C$ will be sufficiently large, perhaps depending
on earlier values of $c,C$, to ensure that certain
inequalities hold. Thus an assertion $\delta\le C\eps^c$, where
$\delta$ depends on $\eps$ and perhaps on certain other parameters in some fashion,
means that there exist $c>0$ and $C<\infty$ such that the inequality holds,
uniformly for all $\eps$ in the relevant range and uniformly in the other
parameters as well.

\medskip
This paper is the first of a series of works 
treating aspects of the meta-question: If the ratio
$\Phi(f)=\|Tf\|_{q}/\|f\|_p$ is large, then what are the 
properties of $f$? The word ``large'' admits various interpretations.
The sequel \cite{christextremal} proves that there exist functions
which extremize $\Phi$, for $(p,q)=((d+1)/d,d+1)$.
In \cite{christxue}, qualitative properties of arbitrary critical
points of $\Phi$ are studied.
The paper \cite{christradonextremals} demonstrates an equivalence
between the inequality studied here, and a certain inequality for
the Radon transform, and explicitly identifies all extremizers for both.

\medskip
I am indebted to Betsy Stovall 
for pointing out the formulation 
\eqref{eq:strongerconclusionA},
\eqref{eq:strongerconclusionB}
of Theorem~\ref{maintheorem},
for innumerable other valuable comments,
and for a thorough proofreading of the manuscript.

\section{Comments}

\subsection{Motivation}
This investigation is motivated by broader considerations.
It is an open problem to determine all the $L^p\to L^q$ inequalities
for all generalized Radon transforms of the type described above.
In many concrete cases, one can guess certain families 
of pairs $(\scripte,\scripte^\star)$ which dictate all the $L^p\to L^q$ inequalities.
One expects that such pairs should fall into finitely many classes,
with each class depending on a small finite number of 
continuous parameters, and that the sets $\scripte,\scripte^\star$ 
should have rather simple geometry.
However, for the general Radon-like transform as described above,
satisfying the condition that $L^p$ is mapped to $L^q$ for some
$q$ strictly greater than $p$, or equivalently (in a localized situation) 
that $L^2$ is mapped to some Sobolev space of finite order,
it is quite unclear how to describe a natural family of such pairs
in terms of $T$ and the associated geometry.
Our second aim is to shed some light on their structure in general,
by examining a basic special case.
Thirdly, and still more speculatively, we hope that 
the development of more refined inequalities
might lead to progress on the basic $L^p\to L^q$ inequalities.

In the corank one case in which both $T$ and its transpose are defined
by integration over one-dimensional manifolds,
the natural pairs are associated to a two-parameter family of
Carnot-Caratheodory balls  in $\scripti$ \cite{taowright}.
For the fundamental example of convolution with the measure $dt$ 
on the curve $(t,t^2,t^3,\cdots,t^d)$ in $\reals^d$, 
an analogue of Theorem~\ref{maintheorem} can be deduced from the analysis
in \cite{ccc}.
More generally, we believe that
a weaker analogue for the general corank one case 
could be deduced from the analysis of
Tao and Wright \cite{taowright}.

\subsection{Symmetries imply a plethora of quasiextremals}
In addition to one-parameter dilation symmetries and rotation 
symmetries (there is a natural action of $O(d-1)$),
our operator enjoys further symmetries which are perhaps less immediately visible.
Adopt coordinates $x=(x',t)$, $x^\star 
= (x'_\star,t^\star)\in\reals^{d-1}\times\reals^d$.
After the substitutions
\begin{equation} \label{changecoords}
(x',t)\mapsto (x',t+|x'|^2),
\qquad
(x'_\star,t^\star)\mapsto (x'_\star,t^\star-|x'_\star|^2),
\end{equation}
The equation $t^\star-t = |x'-x'_\star|^2$
for the incidence manifold
becomes 
$t^\star-t = 2x'\cdot x'_\star$.
In these new coordinates there are manifest symmetries
\begin{equation}
\label{manifestsymmetries}
(x',t)\mapsto (Ax',t),
\qquad
(x'_\star,t^\star)\mapsto
((A^*)^{-1}x'_\star,t^\star)
\end{equation}
where $A$ is any invertible linear endomorphism of $\reals^{d-1}$,
and $A^*$ is its transpose.
The group of all such symmetries is described in greater detail
in \cite{christextremal}.

Closely related is a certain
degeneracy enjoyed by $\scripti$. Namely, for any $1\le k\le d-1$,
there exist manifolds $Y,Y^\star$ of $\reals^d$, of dimensions
$k$ and $d-1-k$ respectively,
such that $Y\times Y^\star\subset \scripti$.
Indeed, 
identify $\reals^d$ with $\reals^{k}\times\reals^{d-1-k}\times\reals^1$,
and take 
$Y=\set{(s;0;-|s|^2): s\in\reals^k}$
and
$Y^\star=\set{(0;t;|t|^2): t\in\reals^{d-1-k}}$.
The rotation symmetry produces large families of such pairs of manifolds
from these.

In this same way one sees that incidence manifolds
$\tilde\scripti$ defined by 
$t-t^\star = \sum_{j=1}^{d-1}c_j|x_j-x^\star_j|^2\}$,
with all $c_j$ nonzero,
are equivalent to $\scripti$ under the action of
${\rm Diff}(\reals^d)\times{\rm Diff}(\reals^d)$;
the signs of the coefficients $c_j$ play no role.

The substitution \eqref{changecoords} is related to
an equivalent description in terms of the Heisenberg group.
$\heis^{d-1}$ can be defined as
a real Lie group of dimension $2d-1$, with coordinates $(y,y^\star,t)
\in\reals^{d-1}\times\reals^{d-1}\times\reals^1$,
for which the left-invariant vector fields are spanned
by $V_j = \p_{y_j}+y^\star_j\p_t$ for $1\le j\le d-1$,
$V_j^\star = \p_{y^\star_j}-y_j\p_t$,
and $T=\p_t$.
The tangent spaces of the level sets of the two
projections  $\pi(y,y^\star,t)=(y,t+y\cdot y^\star)$
and $\pi^\star(y,y^\star,t)=(y^\star,t-y\cdot y^\star)$
of $\heis^{d-1}$ onto $\reals^d$
are spanned by $\{V^\star_j\},\set{V_j}$, respectively.
$\heis^{d-1}$ embeds into $\reals^d\times\reals^d$ via 
$\pi\times\pi^\star$
and is thereby identified with the incidence manifold.
This geometric structure is precisely the one described above.

In this model, pairs of manifolds $Y\subset\reals^d$, $Y^\star\subset\reals^d$ 
with $Y\times Y^\star\subset\scripti\simeq\heis^{d-1}$
have a natural connection with the Lie algebra structure.
If $\frakv\subset\Span\{V_j\}$ and $\frakv^\star\subset\Span\{V_i^\star\}$
are vector subspaces satisfying $[\frakv,\frakv^\star]=0$,
then their images $Y,Y^\star$ 
under the exponential map form such a pair. Moreover, for any $\frakv$,
the dimension of its commutator is $d-1-\dim(\frakv)$.

$\reals^{d-1}\times\reals^{d-1}$ 
has a natural symplectic structure, and
its linear symplectic
automorphisms act naturally on $\heis^{d-1}$ via group automorphisms. 
A certain subgroup acts on $\reals^d\times\reals^d$ by
transformations which leave invariant the incidence manifold
$\scripti$, as described by \eqref{manifestsymmetries}. 
These and other linear symmetries of $\reals^d\times\reals^d$
which preserve $\scripti$,
such as dilations and joint translations in the original coordinate system,
produce all of the quasiextremals described 
in Definition~\ref{balldefinition} from a single quasiextremal.

\subsection{A generalization}
Our operator is prototypical of a class of 
Radon-like transforms characterized by a certain nondegeneracy property. 
Suppose that $\scripti\subset\reals^{d+d}$ is a smooth
manifold of dimension $2d-1$ equipped with submersions $\pi,\pi^\star$ 
mapping $\scripti$ to the two factors $\reals^d$. 
Suppose that the two foliations of  $\scripti$ defined by $\pi,\pi^\star$ are
transverse to one another.
We work only in a sufficiently small relatively compact
subset of $\reals^{d+d}$.
The incidence manifold $\scripti$ is foliated by two transverse
families of $d-1$-dimensional leaves, the level sets of $\pi,\pi^\star$. 
For each $z\in\scripti$ let $T_z,T^\star_z$ denote the tangent spaces to these
leaves, respectively.
Choose a nowhere-vanishing one form $\eta$ on $\scripti$ that annihilates
$T_z+T^\star_z$ at each $z\in\scripti$.
Then $(V,W)\mapsto \eta([V,W])$ defines a skew-symmetric bilinear
form on each subspace $T_z+T^\star_z$. (To define $\eta([V,W])$,
extend $V,V^\star$ to sections in a neighborhood of $z$, form the Lie
bracket, and evaluate; the result is independent of the choices of extensions.)
The general class of operators we have in mind is characterized by the nondegeneracy
of this bilinear form. 

For the generic incidence structure enjoying this nondegeneracy property,
the family of all quasiextremals ought to be smaller, in some natural sense,
than for the particular one studied here.
For such a geometric structure, 
for any manifolds $Y,Y^\star
\subset\reals^d$ 
satisfying $Y\times Y^\star\subset\scripti$,
the sum of the dimensions of $Y,Y^\star$ cannot exceed $d-1$.
For generic structures there exist no such pairs $Y,Y^\star$,
each having strictly positive dimension, with dimensions summing to $d-1$. 
In particular, this is so for another basic example,
convolution with surface measure on the unit sphere in $\reals^d$,
in which $\scripti=\{(x,x^\star)\in\reals^{d+d}: |x-x^\star|=1\}$.
In this case there exist such pairs satisfying $\dim(Y)+\dim(Y^\star)=d-2$,
but not $d-1$.
Stovall \cite{betsyspheres} 
has extended the method of this paper to characterize quasiextremals
for the corresponding inequality for that operator, 
and has found that quasiextremals there, while still numerous, are 
in a natural sense in one-to-one correspondence with a {\em proper} subset
of the set of all quasiextremals here.

\section{Parametrization of subsets of $E,E^\star$}

We now begin the proof of Theorem~\ref{maintheorem}.
Let $E,E^\star\subset\reals^d$ be measurable sets having finite, positive measures.
Define
$\alpha,\alpha_{\star}$  by
\begin{equation}
\alpha|E| = \alpha_{\star}|E^{\star}| = \scriptt(E,E^{\star}).
\end{equation}
As was emphasized in \cite{ccc}, these average numbers of incidences
play a fundamental role in this type of problem, as they do in discrete analogues.
In the case where $\pi,\pi^{\star}$ both have corank one,
Tao and Wright \cite{taowright} observed that $\alpha,\alpha_{\star}$
can be directly interpreted as radii of Carnot-Caratheodory
balls in $\scripti$. In the present situation, the ``balls'' $\scriptb\subset \scripti$
are no longer determined by their centers $\bar z$
and these two parameters; for $d>2$ there is quite a bit of additional freedom.

\begin{lemma} \label{lemma:tower}
There exist a point $\bar x\in E$, a measurable set $\Omega_1\subset\reals^{d-1}$,
and a measurable set $\Omega\subset\Omega_1\times \reals^{d-1}$
such that
\begin{gather}
|\Omega_1|= c\alpha
\\
\bar x-(s,|s|^2)\in E^\star 
\text{ for each } s\in\Omega_1
\\
|\set{t: (s,t)\in\Omega}| = c\alpha_{\star} \ \text{ for each } s\in \Omega_1
\\
\label{Einclusion}
\bar x-(s,|s|^2)+(t,|t|^2)\in E
\text{ for each } (s,t)\in \Omega.
\end{gather}
\end{lemma}
Here $c>0$ is a constant, independent of $E,E^\star,\alpha,\alpha_{\star}$. 
For the proof of Lemma~\ref{lemma:tower} see \cite{ccc}.
The roles of $E,E^\star$ in this lemma can be reversed, thus
producing certain subsets of $E^\star$.

Define
\begin{gather}
\label{tildeOmegadefn}
\tilde\Omega = \{(s,u): (s,s+u)\in\Omega\}
\\
\scriptf(s)=\{u: (s,u)\in\tilde\Omega\}.
\end{gather}
Then $|\scriptf(s)|=c\alpha_{\star}$ for all $s\in\Omega_1$.
Making the change of variables $t=u+s$,
\begin{equation}
-(s,|s|^2)+(t,|t|^2)
= (u,2s\cdot u + |u|^2) = \Psi(s,u).
\end{equation}
Define
$H(u,r) = (u,\tfrac12(r-|u|^2))$ and $\tilde E = H(E)$; then 
$|E| = 2|\tilde E|$. Defining
\begin{equation}\label{Phidefn}
\Phi(s,u) = (u,s\cdot u),
\end{equation}
we have $H\circ\Psi = \Phi$ and therefore, by \eqref{Einclusion},
\begin{equation}|E|\ge 2|\Phi(\tilde\Omega)|.\end{equation}

Following the strategy of \cite{ccc},
tather than seeking an upper bound for $\scriptt(E,E^\star)$
directly in terms of the measures of $E,E^\star$,
we will establish a lower bound on $|E|$ of the form
\begin{equation} \label{imagelowerbound}
|\Phi(\tilde\Omega)|
\ge c\alpha_{\star}^{d/(d-1)}\alpha^{1/(d-1)}.
\end{equation}
Since $\Phi(\tilde\Omega)\subset E$,
this implies that $|E|
\ge c\alpha_{\star}^{d/(d-1)}\alpha^{1/(d-1)}$.
By invoking the definitions of $\alpha,\alpha_{\star}$ 
one finds that this is equivalent to the endpoint restricted weak type inequality
$\scriptt(E,E^\star)\le C|E|^{d/(d+1)}|E^\star|^{d/(d+1)}$. 

\section{Slicing bound}

For polynomial mappings between spaces of equal dimensions, a bound
for $|\Phi(\tilde\Omega)|$
can be obtained \cite{ccc},\cite{taowright},\cite{christerdogan}
simply by writing $|\Phi(\tilde\Omega)|\ge c
\int_{\tilde\Omega} |J|$, where $J$ is the Jacobian determinant of $\Phi$
and $c$ is a positive constant, depending on $\Phi$, 
which takes into account the failure of $\Phi$ to be injective.
The basic difficulty in establishing any lower bound on $|\Phi(\tilde\Omega)|$, 
from this perspective, is that $\Phi$ maps a space of
dimension $2d-2$ to a space of lower (if $d>2$) dimension $d$.

In non-equidimensional circumstances, a simple way to obtain a bound 
is via a slicing argument, as was done in  \cite{ccc}. 
One chooses some submanifold $M$ of the domain of $\Phi$ having
the same dimension as the range of $\Phi$, and has the trivial
bound $|\Phi(\tilde\Omega)|\ge|\Phi(M\cap\tilde\Omega)|$;
the latter can then be analyzed by integrating the associated Jacobian.
One bound obtainable for the present situation via slicing is as follows.

\begin{lemma}[Slicing Lemma] \label{lemma:slice}
Let $B\subset\reals^d$ be the (open) unit ball,
and let $\Phi:\reals^{d-1}\times\reals^{d-1}\to\reals^{d-1}\times\reals^1$
be the mapping $\Phi(s,u) = (u,s\cdot u)$.
Let $A:\reals^{d-1}\to\reals^{d-1}$ be a symmetric invertible linear transformation.
Suppose that $\omega\subset A(B)\times\reals^{d-1}$.
Then 
\begin{equation} \label{slicebound}
|\Phi(\omega)| 
\ge c |\det A|^{-1}\int_\omega |Au|\,du\,ds.
\end{equation}
\end{lemma}

\begin{proof}
Make the change of variables $s=At$, $u=A{}^{-1}v$, recalling that
$A$ is symmetric. 
Then $\Phi(s,u) = \tilde A\Phi(t,v)$ where
$\tilde A(y,r) = (A{}^{-1}y, r)$.
Therefore $|\Phi(\omega)| = |\det A|^{-1}\,|\Phi(\tilde\omega)|$
where $\tilde\omega = \{(t,v): (A^{-1}t,Av)\in\omega\}$.

Now $\tilde\omega\subset B\times\reals^{d-1}$.
Let $\nu\in\reals^{d-1}$ be any unit vector, and let 
$a\in\reals^{d-1}$ be any vector orthogonal to $\nu$.
Consider the mapping $\reals\times\reals^{d-1}\owns (r,v)\mapsto \Phi(a+r\nu,v)
\in\reals^{d-1}\times\reals^1$.
The image of $\tilde\omega_{a,\nu}
=\{(r,v): (a+r\nu,v)\in\tilde\omega\}$ under this mapping
lies in $\Phi(\tilde\omega)$, and this mapping is generically injective,
so since its Jacobian determinant equals $v\cdot\nu$,
\begin{equation}
|\Phi(\tilde\omega)| \ge \int_{\tilde\omega_{a,\nu}} |v\cdot\nu|\,dv\,dr.
\end{equation}
This holds for any $a\in \nu^\perp$; averaging over all $a\in B\cap \nu^\perp$
yields the bound
\begin{equation}
|\Phi(\tilde\omega)| \ge c\int_{\tilde\omega} |v\cdot\nu|\,dv\,dt.
\end{equation}
Averaging over all unit vectors $\nu$ gives
\begin{equation}
|\Phi(\tilde\omega)| \ge c\int_{\tilde\omega} |v|\,dv\,dt,
\end{equation}
from which the desired conclusion follows by reversing the change of variables.
\end{proof}

By itself, this bound is inadequate. For one thing, it is not given
that any sizable portion of $\Omega_1$ lies in any ellipsoid of controlled volume.
But there is an even more fundamental obstacle to the use of Lemma~\ref{lemma:slice}. 
Imagine that $|\Omega_1|=1$, that $\Omega_1$
is a subset of a Euclidean ball $B$ of radius $R\gg 1$,
and that $\Omega_1$ is rather evenly distributed throughout $B$,
up to some small spatial scale. 
Inequality \eqref{slicebound} then incorporates a factor of $R^{-(d-1)}$
resulting from the factor $|\det T|^{-1}$; it yields a weaker
bound as $R$ increases. But
according to our main theorem and the intuition underlying it, 
such a situation should be progressively farther from extremal
as $R$ increases, so we seek bounds which improve rather than
worsening as $R\to\infty$.
In contrast, the factor $|Au|$ in \eqref{slicebound}
does have the desired effect, penalizing $\omega$ (by guaranteeing
an improved lower bound for $|\Phi(\omega)|$ and hence
ultimately for $|E|$) if the variable $u$ is not mainly confined
to an appropriate ellipsoid.
If $A$ is $R$ times the identity where $R$ is large,
then for $d>2$, the factor of $R^1$ gained through the expression $|Au|$
is more than offset by the loss of $R^{-(d-1)}$ through $|\det A|^{-1}$.

In \S\ref{section:inflation} we will establish 
a second type of bound, which yields complementary information.
Each suffers from defects, but together they lead to the theorem.

\section{Approximation by convex sets}
In a sense appropriate for our purposes, any set in $\reals^n$ having
finite Lebesgue measure can be well approximated by a convex set,
that is, by an ellipsoid. 

\begin{lemma} \label{lemma:convexifyunbalanced}
For any $n\ge 1$
and $\etta>0$,
there exists $c>0$ with the following property.
For any Lebesgue measurable set $S\subset\reals^n$
satisfying $0<|S|<\infty$
there exists a bounded convex set $\scriptc\subset\reals^n$
so that
for any convex set $\scriptc'\subset\scriptc$,
\begin{equation} \label{convexexclusionunbalanced}
|\scriptc'|\le \tfrac12|\scriptc|
\Rightarrow |S\cap(\scriptc\setminus\scriptc')|
\ge c_0(|S|/|\scriptc|)^\etta |S|.
\end{equation}
\end{lemma}
\noindent
It follows from \eqref{convexexclusionunbalanced} that
$|\scriptc|\ge c_\etta |S|$.
This result is a descendant of an idea of Tao and Wright \cite{taowright},
formulated originally in dimension one,
sharpened in \cite{christerdogan}, 
and generalized here to higher dimensions. The relevance of convex sets here
is an attribute of the particular operators studied in this paper; other sets
must play the corresponding role for other operators. Some related comments
are made in \S\ref{section:conjectures}.

We will require a variant.
A convex set $\scriptc\subset\reals^n$ is said to be balanced
if $x\in \scriptc\Rightarrow -x\in\scriptc$.
\begin{lemma} \label{lemma:convexify}
For any $n\ge 1$ and $\etta>0$, 
there exists $c>0$ with the following property.
For any Lebesgue measurable set $S\subset\reals^n$
satisfying $0<|S|<\infty$
there exists a bounded balanced convex set $\scriptc\subset\reals^n$
so that
for any balanced convex set $\scriptc'\subset\scriptc$,
\begin{equation} \label{convexexclusion}
|\scriptc'|\le \tfrac12|\scriptc|
\Rightarrow |S\cap(\scriptc\setminus\scriptc')|
\ge c(|S|/|\scriptc|)^\etta |S|.
\end{equation}
\end{lemma}
\noindent
As above, it follows that $|\scriptc|\ge c_\etta |S|$.

\begin{proof} 
For Lemma~\ref{lemma:convexifyunbalanced}, begin with some bounded convex set $C$ satisfying
$|C\cap S|\ge\tfrac34|S|$, with
$|C|=2^m|S|$ for some nonnegative integer $m$.
Let $c_0>0$ be a sufficiently small constant, to be determined.

Consider this stopping-time process:
If there exists no convex subset $C'\subset C$ satisfying $|C'|=\tfrac12 |C|$
with  $|S\cap C'|\ge (1-c_0 2^{-\eta m})|S\cap C|$,
then stop. Otherwise replace $C$ by $C'$ and $m$ by $m-1$, and repeat.

This process must stop at some $m\ge 0$.
For if we ever reach the stage $m=0$, the process then stops unless
there exists a convex set $C'$ satisfying both $|C'|=\tfrac12 |S|$
and $|S\cap C'| \ge \prod_{k=0}^\infty (1-c_02^{-\eta k})\tfrac34 |S|$.
Thus $\tfrac12\ge \tfrac34 \prod_{k=0}^\infty (1-c_02^{-\eta k})$.
This is impossible
if $c_0$ is chosen to be a sufficiently small function of $\eta$. 
\end{proof}

\section{Inflation bound} \label{section:inflation}

The material in this section, taken from \cite{edinburgh}, yields a 
short, direct proof of the restricted weak type inequality \eqref{restrictedweaktype}.
It does not by itself suffice to characterize quasiextremals,
but will be one essential ingredient in the analysis.
See also Schlag \cite{schlag} for a related 
discrete combinatorial approach to the inequality.

Write $\uarrow=(u_1,\cdots,u_{d-1})$ to denote a point of $(\reals^{d-1})^{d-1}$.
Form the set
\begin{equation}
\Omega^\natural
=\{(s,\uarrow)\in (\reals^{d-1})^d:
(s,u_i)\in\tilde\Omega\ \forall\, 1\le i\le d-1.
\}
\end{equation}
Define $\Psi:(\reals^{d-1})^d\to (\reals^d)^{d-1}$ by
\begin{equation}
\Psi(s,\uarrow) = ((u_1,s\cdot u_1),(u_2,s\cdot u_2),\cdots,(u_{d-1},s\cdot u_{d-1})).
\end{equation}
Then 
\begin{equation}
\Psi(\Omega^\natural)\subset (\Phi(\tilde\Omega))^{d-1}\subset  \tilde E^{d-1}.
\end{equation}
Both the domain and range of $\Psi$ have dimension $d(d-1)$.

$\Psi$ is injective outside a set of measure zero,
its Jacobian determinant is $|\det(\uarrow)|$,
and 
\begin{equation}
|\Psi(\Omega^\natural)| 
= \int_{s\in \Omega_1}
\int_{\uarrow\in\scriptf(s)^{d-1}}
|\det(\uarrow)| 
\,d\uarrow\,ds.
\end{equation}

\begin{lemma} \label{lemma:detbound}
Let $\scriptc\subset\reals^n$ be a bounded, balanced convex set.
Let $\mu$ be a positive, finite measure supported on $\scriptc$.
Suppose that for any balanced convex subset $\scriptc'\subset\scriptc$
satisfying $|\scriptc'|\le\delta|\scriptc|$,
$\mu(\scriptc\setminus\scriptc')\ge \lambda$.
Then
\begin{equation}
\int_{\scriptc^n} |\det(\uarrow)|\,\prod_{i=1}^n \,d\mu(u_i)
\ge c\delta^{n}\lambda^n|\scriptc|
\end{equation}
where $c>0$ depends only on $n$.
\end{lemma}
\noindent The power of $\delta$ here is not optimal, but the precise
dependence on $\delta$ is unimportant for us.

\begin{proof}
By applying an affine change of coordinates in $\reals^n$, we may
reduce to the case where $\scriptc$ is the unit ball; 
the factor $|\scriptc|$ in the conclusion results from the Jacobian of
this change of variables and the transformation law for $|\det(\uarrow)|$.

Write $|\det(\uarrow)| = \prod_{i=1}^{d-1} \dist(u_i,V_{i-1})$
where $V_0=\{0\}$, $V_i=\Span\{u_1,\cdots,u_i\}$,
and $\dist(v,V)$ denotes the distance from $v$ to $V$.
Fixing $(u_1,\cdots,u_{n-1})$,
define $\scriptc'$ to be the set of all $u_n$ satisfying
$\dist(u_n,\Span(u_1,\cdots,u_{n-1}))<c_n\delta$, where $c_n$ is a 
constant chosen sufficiently small to ensure
that $|\scriptc'|\le\tfrac12 |\scriptc|$.
Since $\scriptc'$ is convex and balanced, 
\begin{equation}
\int_\scriptc \dist(u_n,\Span(u_1,\cdots,u_{n-1}))\,d\mu(u_n)
\ge c\mu(\scriptc\setminus\scriptc')
\ge c\delta\lambda.
\end{equation}

Next repeat the argument: Holding $(u_1,\cdots,u_{n-2})$ fixed,
redefine $\scriptc'$ to be the set of all $u_{n-1}$
satisfying $\dist(u_{n-1},\Span(u_1,\cdots,u_{n-2}))\le c_n\delta$,
for another sufficiently constant $c_n$.
The same reasoning as above gives
\begin{equation}
\int_\scriptc \dist(u_{n-1},\Span(u_1,\cdots,u_{n-2}))\,d\mu(u_{n-1})
\ge c\delta\lambda.
\end{equation}
Repeating this reasoning $n$ times results in the desired bound.
\end{proof}

Now for each $s\in\Omega_1$, apply Lemma~\ref{lemma:convexify}
to $\scriptf(s)$ to obtain a balanced convex set $\scriptc(s)\subset\reals^{d-1}$
of measure $\sim 2^{2m(s)}\alpha_{\star}$ for some nonnegative integer $m(s)$,
so that for any convex balanced subset $\scriptc'\subset\scriptc(s)$
of measure $\le\tfrac12|\scriptc(s)|$,
$|\scriptf(s)\cap (\scriptc(s)\setminus\scriptc')| 
\ge c_\etta 2^{-\etta m(s)}\alpha_{\star}$.
Lemma~\ref{lemma:detbound} (applied with $ \mu$ equal to Lebesgue measure
restricted to $\scriptf(s)\cap\scriptc(s)$) yields the lower bound
\begin{multline}
\int_{(\scriptf(s)\cap \scriptc(s))^{d-1}}
|\det(\uarrow)|\,d\uarrow
\ge c_\etta |\scriptf(s)\cap\scriptc(s)|^{d-1} |\scriptc(s) |
\\
\ge c_\etta 2^{2m(s)}2^{-(d-1)\etta m(s)}|\scriptf(s)|^d
\sim 2^{m(s)}\alpha_{\star}^d
\end{multline}
if we define $\etta = (d-1)^{-1}$.
We thus conclude that
\begin{equation}
|\Psi(\Omega^\natural)| 
\ge c  \alpha_{\star}^d \int_{\Omega_1} 2^{m(s)}\,ds
\ge c  \alpha\alpha_{\star}^d.
\end{equation}

We have proved
\begin{lemma} \label{lemma:imagebound}
Let $E,\alpha,\alpha_{\star},\Phi$ and $\bar x,\Omega_1,\Omega$ 
satisfy the conclusions of Lemma~\ref{lemma:tower}.
Define $\Phi,\tilde\Omega$ as in \eqref{tildeOmegadefn},\eqref{Phidefn}.
Then
$|\Phi(\tilde\Omega)|\ge c\alpha_{\star}^{d/(d-1)}\alpha^{1/(d-1)}$.
\end{lemma}

The conclusion implies \eqref{restrictedweaktype}.
Moreover, unless $m(s)$ is small for most $s\in\Omega_1$,
we obtain an improved bound, which implies that 
if $T(E,E^\star)\ge \eps |E|^{d/(d+1)}|E^\star|^{d/(d+1)}$, then 
\begin{equation}
|\Omega_1|^{-1}\int_{\Omega_1} 2^{m(s)}\,ds
\le C\eps^{-C}.
\end{equation}
Thus roughly speaking, the typical set $\scriptf(s)$ has a subset 
of measure $\sim\alpha_{\star}$ that is contained in a convex balanced set $\scriptc(s)$
of measure $\lesssim \eps^{-C}\alpha_{\star}$.

From the point of view of our main theorem, this conclusion is defective
in two respects.
Firstly no geometric conclusion on $\Omega_1$ is obtained; 
however, we will see momentarily that this is easily remedied.
Secondly, and more significantly,
no relation between the different sets $\scriptc(s)$ is implied.
We need to show that $\cup_s \scriptc(s)$ is comparable
to a convex balanced set of measure $\lesssim \eps^{-C}$; 
and that this convex set is appropriately related to a convex 
set to which $\Omega_1$ is comparable.

\section{Merging the inflation and slicing bounds}

\begin{lemma} \label{lemma:merging}
There exists an exponent $b<\infty$ with the following property.
Let $\eps>0$ and let $(E,E^\star)$ be an $\eps$-quasiextremal pair.
Define $\alpha = \scriptt(E,E^\star)/|E|$ and $\alpha_{\star} = \scriptt(E,E^\star)/|E^\star|$.
Then there exist a point $\bar x\in E$, a measurable set $\Omega_1\subset\reals^{d-1}$,
a measurable set $\Omega\subset\Omega_1\times \reals^{d-1}$,
and a convex set $\scriptc\subset\reals^{d-1}$ having finite Lebesgue measure,
such that
\begin{gather}
\Omega_1\subset\scriptc
\\
|\Omega_1|= c\alpha
\\
|\scriptc|\le C\eps^{-b}\alpha
\\
\bar x-(s,|s|^2)\in E^\star 
\text{ for each } s\in\Omega_1
\\
|\set{t: (s,t)\in\Omega}| = c\alpha_{\star} \ \text{ for each } s\in \Omega_1
\\
\bar x-(s,|s|^2)+(t,|t|^2)\in E
\text{ for each } (s,t)\in \Omega.
\end{gather}
Moreover, there exists $\bar s\in\reals^{d-1}$ such the translated
convex set $\scriptc-\bar s$ is balanced.
\end{lemma}

\begin{proof}
By the same reasoning already used above,
there exist $\bar x_\star\in E^\star$
and sets $\omega_1\subset\reals^d$,
$\omega_2\subset\omega_1\times\reals^d$,
$\omega_3\subset\omega_2\times\reals^d$
with the following properties:
\begin{gather}
\label{mergingprooffirst}
\bar x_\star+(r,|r|^2)\in E
\ \forall\,r\in\omega_1
\\
\bar x_\star+(r,|r|^2)-(s,|s|^2)\in E^\star
\ \forall\,(r,s)\in\omega_2
\\
\bar x_\star+(r,|r|^2)-(s,|s|^2)+ (t,|t|^2)\in E
\ \forall\,(r,s,t)\in\omega_3
\\
|\omega_1| = c\alpha_{\star}
\\
|\set{s: (r,s)\in\omega_2}| = c\alpha\ \text{ for each } r\in \omega_1
\\
|\set{t: (r,s,t)\in\omega_3}| = c\alpha_{\star}\ 
\text{ for each } (r,s)\in \omega_2.
\label{mergingprooflast}
\end{gather}
Suppose that the pair $(E,E^\star)$ is $\eps$-quasiextremal.
By considering $\omega_2$ and invoking the conclusion of \S\ref{section:inflation}
we conclude that there exist $\bar r\in\omega_1$
and a convex balanced set $\scriptc$ centered at $\bar r $ 
such that  $|\scriptc|\lesssim \eps^{-C}\alpha$
and
$|\scriptc\cap \{s: (\bar r,s)\in\omega_2\}|\ge c\alpha$.
Now set $\bar x = \bar x_\star + (\bar r,|\bar r|^2)$,
$\Omega_1=\{s: (\bar r,s)\in\omega_2\}$,
and
$\Omega=\{t: (\bar r,s,t)\in\omega_3\}$.
\end{proof}

We now prove the main result, Theorem~\ref{maintheorem}. 
Let $(E,E^\star)$ be an $\eps$-quasiextremal pair.
Let $\scriptc\subset\reals^{d-1}$ be a convex set 
satisfying the conclusions of Lemma~\ref{lemma:merging}.
There exists an ellipsoid which contains $\scriptc$ and has measure
comparable to that of $\scriptc$, up to a factor which depends only
on the dimension $d$. This ellipsoid equals $A(B)$
for a certain invertible symmetric linear transformation $A$
of $\reals^{d-1}$, where $B$ is the unit ball.
Thus  $|\det A|\sim |\scriptc|$.

By Lemma~\ref{lemma:slice},
\begin{align*}
|E| 
&\ge c |\det A|^{-1} \int_{\tilde\Omega} |A(u)|\,du\,ds
\\
&= c |\det A|^{-1} \int_{s\in \Omega_1}\int_{\scriptf(s)} |A(u)|\,du\,ds
\\
&= c |\det A|^{-2} \int_{\Omega_1}
\int_{\tilde\scriptf(s)} |w|\,dw\,ds
\end{align*}
where $w = A(u)$ ranges over the set 
$\tilde\scriptf(s)=A\scriptf(s)\subset\reals^{d-1}$,
and 
\[|\tilde\scriptf(s)|\sim |\det A|\alpha_{\star} \sim |\scriptc|\alpha_{\star}.\]
By passing to a subset of $\Omega$, we can assume that 
all sets $|\scriptf(s)|$ have the same measures, hence
that 
$|\tilde\scriptf(s)|= c|\scriptc|\alpha_\star$
for all $s\in\Omega_1$, for a certain small constant $c>0$.

Clearly $\int_S |w|\,dw\gtrsim |S|^{d/(d-1)}$
for any Lebesgue measurable set $S\subset\reals^{d-1}$.
Therefore
\begin{equation} \label{slicinginfo}
\int_{\tilde\scriptf(s)} |w|\,dw
\ge c |\tilde\scriptf(s)|^{d/(d-1)}
\sim |\det A|^{d/(d-1)} \alpha_{\star}^{d/(d-1)}.
\end{equation}
An equally evident strengthened version of this bound will be the key
to constraining the structure of $\Omega$:
For any $\rho\ge |\tilde\scriptf(s)|^{1/(d-1)}$, 
either
\begin{equation} \label{eq:dichotomy1}
\int_{\tilde\scriptf(s)} |w|\,dw
\ge c
\frac{\rho} { |\tilde\scriptf(s)|^{1/(d-1)} }
|\det A|^{d/(d-1)} \alpha_{\star}^{d/(d-1)},
\end{equation}
or 
\begin{equation} \label{eq:dichotomy2}
| \tilde\scriptf(s)\cap B(0,\rho) | 
\ge c'\alpha_{\star}|\det A|
\end{equation}
for a certain constant $c'>0$ independent of $\rho$, where
$B(0,\rho)\subset\reals^{d-1}$ denotes the ball of radius $\rho$
centered at the origin.

From the cruder conclusion \eqref{slicinginfo} we deduce already that
\begin{multline} \label{Elowerbound}
|E| \ge c|\det A|^{-2}|\det A|^{d/(d-1)}\alpha\alpha_{\star}^{d/(d-1)}
\\
\sim  |\scriptc|^{-(d-2)/(d-1)}\alpha\alpha_{\star}^{d/(d-1)}
\ge c \eps^{b(d-2)/(d-1)}\alpha^{1/(d-1)}\alpha_{\star}^{d/(d-1)}.
\end{multline}
From this and the definitions of $\alpha,\alpha_{\star}$ it follows by
a bit of algebra that
\[
\scriptt(E,E^\star)\le C\eps^{-C}|E|^{d/(d+1)}|E^\star|^{d/(d+1)}.
\]
But this, together with the $\eps$-quasiextremality hypothesis 
that
$\scriptt(E,E^\star)$ is $\ge \eps|E|^{d/(d+1)}|E^\star|^{d/(d+1)}$,
forces an upper bound on $\eps$,
independent of $E,E^\star$. Thus we once again recover the restricted weak
type endpoint inequality
$\scriptt(E,E^\star)\le C|E|^{d/(d+1)}|E^\star|^{d/(d+1)}$.

To squeeze out new information, 
apply the dichotomy \eqref{eq:dichotomy1},\eqref{eq:dichotomy2} with
\begin{equation}
\rho = \lambda\eps^{-a}|\tilde\scriptf(s)|^{1/(d-1)}, 
\end{equation}
where $a>0$ and $\lambda\gg 1$ are constants to be specified below.
Then either
\begin{enumerate}
\item
There exists a subset $\Omega_1^\dagger\subset\Omega_1$ of measure $\ge c\alpha$
such that for each $s\in\Omega_1^\dagger$,
$|\tilde\scriptf(s)\cap B(0,\lambda\eps^{-a})|\ge c\alpha_{\star}|\det A|$,
or
\item
There exists a subset $\Omega_1^\ddagger$ of measure $\ge c\alpha$
such that for each $s\in\Omega_1^\dagger$,
$\int_{\tilde\scriptf(s)} |w|\,dw \ge c\lambda\eps^{-a}|\det A|^{d/(d-1)}\alpha_{\star}^{d/(d-1)}$.
\end{enumerate}
In case (2),
by integrating over $\Omega_1^\ddagger$ we conclude that
\begin{equation}
|E| \ge 
c \lambda\eps^{-a+b(d-2)/(d-1)}\alpha^{1/(d-1)}\alpha_{\star}^{d/(d-1)}
\end{equation}
and thence, by choosing $a> b(d-2)/(d-1)$, that
\begin{equation}
\scriptt(E,E^\star)\le C\lambda^{-a'}\eps^{\gamma}|E|^{d/(d+1)}|E^\star|^{d/(d+1)}
\end{equation}
for some exponents $a',\gamma>0$.
The exponent $a$ can be chosen so that $\gamma=1$.
Here $C$ is independent of $\lambda,\eps,a'$.
Choose $\lambda$ sufficiently large that this contradicts 
the quasiextremality hypothesis
$\scriptt(E,E^\star)\ge \eps|E|^{d/(d+1)}|E^\star|^{d/(d+1)}$.
Therefore case (2) cannot arise; case (1) must hold.
Henceforth $\lambda,a$ and hence $\rho$ remain fixed.

In case (1), $\tilde E$ contains
$\Phi(\{(s,u)\in\Omega: s\in\Omega_1^\dagger
\text{ and } u\in A^{-1}(B(0,\rho)\})$.
The same reasoning that established \eqref{Elowerbound} 
proves that
this subset of $\tilde E$ has measure
$\ge c\eps^C \alpha^{1/(d-1)}\alpha_{\star}^{d/(d-1)}$.
Reversing the change of variables that transformed $E$ to $\tilde E$,
and unraveling notation,
we conclude that 
\[|E\cap \pi\scriptb(\bar z, \earrow,r,r^{\star})|\ge c\eps^C|E|\]
where
$\bar z =(\bar x,\bar y)$
with $\bar y = \bar x - (\bar s,|\bar s|^2)$,
$\bar s\in\reals^{d-1}$ is a point such that the convex set $\scriptc-\bar s$
is balanced,
and the elements $e_j$ of the orthonormal basis $\earrow$
and components $r_j$ of $r$
are eigenvectors and eigenvalues of $A$.

The sets $E,E^\star$ play symmetric roles, so it follows in
exactly the same way that $E^\star$ is related to $\pi^\star(\scriptb')$,
for some other ``ball'' $\scriptb'$, in the same way that
$E$ is related to $\pi(\scriptb)$.
It remains to show that $\scriptb,\scriptb'$
can be taken to be equal, after possibly enlarging 
the parameters $\rho,r_j,r^\star_j$ in their definitions
by a factor $C\eps^{-C}$.
This follows from information already brought out.

Indeed, it has been shown that there exist $\bar x_\star$
and sets $\omega_1,\omega_2,\omega_3$ as in the
proof of Lemma~\ref{lemma:merging},
together with convex balanced sets $\scriptc_1,\scriptc_2,\scriptc_3
\subset\reals^{d-1}$ and
a parameter $\bar r\in\reals^{d-1}$
such that $\omega_1\subset \bar r+\scriptc_1$,
and whenever $(r,s,t)\in\omega_3$,
$s-r\in\scriptc_2$ and $t-s\in\scriptc_3$. 
Both $|\scriptc_1|$
and $|\scriptc_3|$ 
are
$\sim C\eps^{-C}\alpha_\star$,
while
$|\scriptc_2|\sim C\eps^{-C}\alpha$.
$\scriptc_2$ is determined by $\scriptc_1$
in the following way:
There exist an orthonormal basis
$\set{e_j: 1\le j\le d-1}$ for $\reals^{d-1}$
and positive real numbers $r_j$ such
that $\scriptc_1$ is comparable to 
$\set{y'\in\reals^{d-1}: |\langle y',e_j\rangle|<r_j \text{ for all } j}$
and $\prod_{j=1}^{d-1}r_j = C\eps^{-C}\alpha_\star$;
we can redefine $\scriptc_1$ to be this set. 
Then $\scriptc_2$ can be taken to be 
$\{y'\in\reals^{d-1}: |\langle y',e_j\rangle|<r_j^\star \text{ for all } j\}$,
where $r_jr_j^\star=\rho$ and $\rho$ is determined 
from $\{r_j\}$ by the requirement that $\prod_j r_j^\star =
C\eps^{-C}\alpha$.
The above analysis shows that $\scriptc_2$ is determined by $\scriptc_1$
in this sense.

Now since $E,E^\star$ play symmetric roles, the same analysis
shows that $\scriptc_3$ is determined by $\scriptc_2$ in the same
way. This forces $\scriptc_3=\scriptc_1$, up to the replacement
of $r_j$ by $C\eps^{-C}r_j$ for each $j$. 
Thus we may take $\scriptc_3$ to equal $\scriptc_1$.

We know that 
\[\bar x_\star+(r,|r|^2)-(s,|s|^2)\in E^\star
\text{ for all $(r,s)\in\omega_2$,}\]
and that 
\[\phi(r,s,t)=\bar x_\star+(r,|r|^2)-(s,|s|^2)+(t,|t|^2)\in E
\text{ for all $(r,s,t)\in\omega_3$.}\]
This produces subsets of $\tilde E\subset E$ and $\tilde E^\star\subset E^\star$
satisfying the desired lower bound 
$\scriptt(\tilde E,\tilde E^\star)\ge c\eps^C \scriptt(E,E^\star)$.
Moreover $\tilde E^\star \subset\pi^\star(\scriptb)$.
Thus all that remains to be shown is that
$\phi(\omega_3)\subset\pi(\scriptb)$ for the same ball $\scriptb$.

By definition of $\scriptb$,
this amounts to showing that
\begin{equation}
\label{compatibility1}
\big|
\phi(r,s,t)_d-[(\bar x_\star)_d + |\phi(r,s,t)'-(\bar x_\star')|^2]
\big|
<C\eps^{-C}\rho
\end{equation}
for all $(r,s,t)\in\omega_3$,
where we have written 
\[\phi(r,s,t) = (\phi(r,s,t)',\phi(r,s,t)_d)
\in\reals^{d-1}\times\reals^1.\]
Substituting the definition
\[\phi(r,s,t) = \bar x_\star
+ (r,|r|^2)-(s,|s|^2)+(t,|t|^2),\]
\eqref{compatibility1} becomes
\begin{equation} \label{compatibility2}
\big|
|r|^2-|s|^2+|t|^2-|r-s+t|^2
\big|
<C\eps^{-C}\rho.
\end{equation}
Since $(t-s)\in\scriptc_1$, $(s-r)\in\scriptc_2$, and
\[ |r|^2-|s|^2+|t|^2-|r-s+t|^2
= 2(t-s)\cdot(s-r),\]
this follows directly from
the duality relationship between $\scriptc_1$ and $\scriptc_2$.

Thus we have shown that
there exists a pair $(B,B^\star)=(\pi(\scriptb),\pi^\star(\scriptb))$
satisfying
\begin{gather} 
\scriptt(E\cap B,E^{\star}\cap B^{\star}) \ge C^{-1}\scriptt(E,E^{\star})
\label{alternativeconclusion1}
\\
\intertext{and}
|B|\le C\eps^{-A}|E| \text{ and } |B^{\star}|\le C\eps^{-A}|E^{\star}|.
\label{alternativeconclusion2}
\end{gather}
This is essentially stronger than the conclusion stated in Theorem~\ref{maintheorem},
as will be shown below using the next lemma.

\begin{lemma} \label{lemma:covering}
There exist $C,A<\infty$ such that for any $\delta\in(0,1]$
and any set $\scriptb =\scriptb(\bar z,\earrow,r,r^{\star})\subset\scripti$ 
of the type described in Definition~\ref{balldefinition},
there exists a family of subsets $\{\scriptb_j: j\in J\}$ of $\scripti$,
each of which is likewise a set of the type described in Definition~\ref{balldefinition},
satisfying
\begin{gather*}
\scriptb=\cup_{j\in J}\scriptb_j,
\\
|J|\le C\delta^{-A},
\\
|\pi(\scriptb_j)|=\delta|\pi(\scriptb)| \text{ for all $j$,}
\\
|\pi^\star(\scriptb_j)|=\delta|\pi^\star(\scriptb)| \text{ for all $j$.}
\end{gather*}
\end{lemma}
\noindent Here $J$ denotes the cardinality of the index set $J$.

\begin{proof}
Symmetries of $\scripti$ 
(cf.\ \eqref{manifestsymmetries})
permit a reduction to the case
where $\bar z=(0,0)$, $\earrow$ is the standard basis for
$\reals^d$, and $\rho=r_i=r_j^\star=1$ for all $i,j$. 
Then
$|\pi(\scriptb)|=|\pi^\star(\scriptb)|$.

Let $\eta=c\delta^{1/(d+1)}$ and $\eta'=c'\delta^{2/(d+1)}$
for constants $c,c'$ to be chosen below.
Let $\{z_j: j\in J\}$ be a finite subset of $\scriptb$
such that $|z_i-z_j|\gtrsim\eta'$ for all $i\ne j$,
and such that for any $z\in\scriptb$
there exists $j$ such that $|z-z_j|\le \eta'$. 
Then $|J|\le C\delta^{-A}$ for some finite constants $C,A$.

Define $\scriptb_j = \scriptb(z_j,\earrow,r,r^\star)$
where $r_k=r^\star_l=\eta$ (and consequently $\rho=\eta^2$)
for all indices $1\le k,l\le d-1$.
Then 
$|\pi(\scriptb_j)|=|\pi^\star(\scriptb_j)|
= C\eta^{d+1} = Cc^{d+1}\delta$
for a certain constant $C$;
in particular, these are independent of $j$. 
There is a unique $c$, independent of $\delta$, such that 
$|\pi(\scriptb_j)|=\delta|\pi(\scriptb)|$
and
$|\pi^\star(\scriptb_j)|=\delta|\pi^\star(\scriptb)|$
for all $j$.
If $c'$ is chosen to be sufficiently small,
then $\cup_j\scriptb_j$ clearly covers $\scriptb$;
the exponent $2/(d+1)$ in the definition of $\eta'$
is essential here because $\rho$ is proportional to $\eta^2$.
\end{proof}

To complete the proof of Theorem~\ref{maintheorem}, 
let $\scriptb$ be as in 
\eqref{alternativeconclusion1},\eqref{alternativeconclusion2}.
Apply the lemma with $\delta = \eps^\Gamma$ for a sufficiently
large exponent $\Gamma$, to obtain 
sets $\scriptb_j$ such that $B_j=\pi(\scriptb_j)$
and
$B_j^\star = \pi^\star(\scriptb_j)$
satisfy $|B_j|\le |E|$ and $|B_j^\star|\le |E^\star|$ for all $j$.
$\Gamma$ can be taken to depend only on the exponent $A$
in \eqref{alternativeconclusion1},\eqref{alternativeconclusion2}.
Since
\begin{align*}
\scriptt(E\cap B,E^\star\cap B^\star)
&= c|\scripti\cap (E\cap B\times E^\star\cap B^\star)|
\\
&\le c\sum_{j\in J} 
|\scripti\cap (E\cap B_j\times E^\star\cap B^\star_j)|
\\
&= 
\sum_{j\in J}
\scriptt(E\cap B_j,E^\star\cap B_j^\star)
\end{align*}
and $|J|\le C\eps^{-C}$,
there must exist an index $j$ for which
\begin{equation*}
\scriptt(E\cap B_j,E^\star\cap B_j^\star)
\ge c\eps^C 
\scriptt(E\cap B,E^\star\cap B^\star)
\ge c\eps^C\scriptt(E,E^\star),
\end{equation*}
as was to be proved.
Here $C$ is determined by $\Gamma$, hence by $A$; it does not
depend on $E,E^\star$.
\qed

\begin{proof}[Proof of Theorem~\ref{thm:reformulation}]
Let $E,E^\star$ be arbitrary measurable sets of strictly
positive Lebesgue measures. If $\scriptt(E,E^\star)=0$ then
there is nothing to prove. Otherwise define 
$\eps>0$ by
\begin{equation} \label{eq:epsilondefn}
\scriptt(E,E^\star) = \eps |E|^{d/(d+1)}|E^\star|^{d/(d+1)}.
\end{equation}
According to Theorem~\ref{maintheorem},
there exists a pair $(B,B^\star)=(\pi(\scriptb),\pi^\star(\scriptb))$
such that  $|B|\le|E|$, $|B^\star|\le |E^\star|$,
and 
\[
\scriptt(E\cap B, E^\star\cap B^\star)\gtrsim
\eps^A
\scriptt(E,E^\star). 
\]
Since 
\[ \scriptt(E\cap B, E^\star\cap B^\star)
\le C|E\cap B|^{d/(d+1)}|E^\star\cap B^\star|^{d/(d+1)},
\]
it follows by a bit of algebra that
\[
\frac{|E\cap B|}{|E|}\cdot \frac{|E^\star\cap B^\star|}{|E^\star|}
\gtrsim \eps^{(A+1)(d+1)/d}.
\]
Substituting this upper bound for $\eps$ into
\eqref{eq:epsilondefn} yields 
\[
\scriptt(E,E^\star)\lesssim
|E|^{d/(d+1)}|E^\star|^{d/(d+1)}
\Big(
\frac{|E\cap B|}{|E|}\cdot \frac{|E^\star\cap B^\star|}{|E^\star|}\Big)^\delta
\]
for a certain $\delta>0$.
\end{proof}


\section{A trilinear variant}
A restricted weak type inequality cannot be extrapolated
to a strong type inequality without additional information.
Our basic bilinear inequality for $\scriptt(E,F)$ admits
the following trilinear variant, which will be the key to 
the extrapolation.

\begin{lemma} \label{lemma:trulymultilinear}
Let $E,E',G\subset\reals^d$
be Lebesgue measurable sets with finite measures.
Suppose that $T(\chi_{E'})(x)\ge\beta'$ for all $x\in G$.
Then
\begin{equation} \label{productbound1}
\big(\scriptt(E,G)|E|^{-1}\big)^{1/(d-1)}
{\beta'}^{d/(d-1)}\le C|E'|.
\end{equation}
\end{lemma}

A more symmetric variant is as follows: If in addition
$T(\chi_E)(x)\ge\beta$ for all $x\in G$, then
$|E'|\ge c \beta^{1/(d-1)}{\beta'}^{d/(d-1)}$.

\begin{proof}[Proof of Lemma~\ref{lemma:trulymultilinear}]
The proof of Lemma~\ref{lemma:tower} yields the following variant.
There exist a point $\bar x\in E$, a measurable set $\Omega_1\subset\reals^{d-1}$,
and a measurable set $\Omega\subset\Omega_1\times \reals^{d-1}$
such that 
\begin{gather}
|\Omega_1|= c\scriptt(E,G)|E|^{-1}
\\
\bar x-(s,|s|^2)\in G
\text{ for each } s\in\Omega_1
\\
|\set{t: (s,t)\in\Omega}| = c\beta' \ \text{ for each } s\in \Omega_1
\\
\bar x-(s,|s|^2)+(t,|t|^2)\in E'
\text{ for each } (s,t)\in \Omega.
\end{gather}
Lemma~\ref{lemma:imagebound} now directly yields the bound \eqref{productbound1}.
\end{proof}

\section{The strong type and Lorentz space inequalities}

Although the strong type $(\frac{d+1}{d},d+1)$ inequality is already known, 
we next show how it can be deduced from an extension of the
above proof of the restricted weak type bound.
This argument will be the basis for our proofs 
of Theorems~\ref{thm:functions} and \ref{thm:Lorentz}. 

Write $p=q=\frac{d+1}{d}$ and
consider functions $f,g\in L^p,L^q$.
By sacrificing a bounded factor we may take
$f = \sum_k 2^k\chi_{E_k}$
and
$g = \sum_j 2^j\chi_{F_j}$
where the sets $E_k$ are pairwise disjoint
and the sets $F_j$  are likewise pairwise disjoint,
and $j,k$ range independently over subsets of $\integers$.
The simple bound for $\sum_{j,k}2^j2^k\scriptt(E_k,F_j)$
obtained directly from the restricted weak type bound 
does not suffice, because a single set $F_j$
could conceivably interact strongly with many $E_k$, in
the sense that $\scriptt(E_k,F_j)\gtrsim |E_k|^{1/p}|F_j|^{1/q}$,
and vice versa. The main idea is to show that this can happen only
in a trivial and harmless way.

Consider first the case of a single index $j$; this amounts
to a weak type $(p,q')$ estimate.
Let $\eps,\eta\in(0,\tfrac12]$
be arbitrary. Suppose that $\sum_k 2^{kp}|E_k|=1$, and that 
\begin{equation}|E_k|\sim \eta 2^{-kp} \text{ for all } k.\end{equation}
Suppose further that 
\begin{equation}
\scriptt(E_k,F)\sim \eps|E_k|^{1/p}|F|^{1/q}
\text{ for all } k.
\end{equation}
Then the number $M$ of indices $k$ is finite, and $M\eta\lesssim 1$.
We suppose that $|k-l|\ge A\log(1/\eps)$ for any two distinct
indices appearing in the sum, where $A$ is a sufficiently large positive constant,
to be specified later in the proof. This will cost a factor of
$C_A\log(1/\eps)$, which will be dealt with below.

Define 
\begin{equation}
G_k=\{x\in F: T\chi_{E_k}(x)\ge c_0\eps |E_k|^{1/p}|F|^{1/q}\cdot|F|^{-1}\},
\end{equation}
where $c_0>0$ is a constant.
If $c_0$ is chosen to be sufficiently small then
$\scriptt(E_k,F\setminus G_k)\le\tfrac12\scriptt(E_k,F)$,
so 
\begin{equation}
\scriptt(E_k,G_k)\sim\scriptt(E_k,F).
\end{equation}
Since $\scriptt(E_k,G_k)\lesssim |E_k|^{1/p}|G_k|^{1/q}$,
this implies that 
\begin{equation}|G_k|\gtrsim \eps^q|F|. \end{equation}

A useful bound is obtained by considering $|F|^{-1}\sum_k|G_k|
= |F|^{-1}\int_F \sum_k\chi_{G_k}$. By H\"older's inequality,
\begin{equation}
\begin{split}
(|F|^{-1}\sum_k|G_k|)^2
&\le |F|^{-1}\int_F (\sum_k\chi_{G_k})^2
\\
&\le |F|^{-1}\sum_k|G_k|\
+ |F|^{-1}\sum_{k\ne l}|G_k\cap G_l|.
\end{split}
\end{equation}
Therefore either $\sum_k|G_k|\lesssim |F|$,
or $ (|F|^{-1}\sum_k|G_k|)^2
\lesssim |F|^{-1} \sum_{k\ne l}|G_k\cap G_l|$.

Let $N$ be the number of indices $k$.
In the second case of this dichotomy,
since $|G_k| \gtrsim \eps^q|F|$,
we conclude that 
\begin{equation}
(N\eps^q)^2
\lesssim (|F|^{-1}\sum_k|G_k|)^2
\lesssim N^2|F|^{-1}\max_{k\ne l}|G_k\cap G_l|,
\end{equation}
so there exists a pair $k\ne l$ such that 
\begin{equation}\label{overlap}|G_k\cap G_l|\gtrsim \eps^{2q}|F|.\end{equation}
We have now arrived at the key step of the proof of the strong type inequality;
we claim that \eqref{overlap} cannot hold for $k\ne l$.
From this it would follow that
\begin{equation}\sum_k|G_k|\lesssim |F|.\end{equation}
The interpretation is that while many sets $E_k$
can interact $\eps$-strongly with a single set $F$
for small $\eps$, they can do so only in a trivial way,
by interacting with essentially pairwise disjoint subsets of $F$.

\begin{proof}[Proof of Claim]
Apply Lemma~\ref{lemma:trulymultilinear}
with $E=E_k$, $E'=E_l$, $G = G_k\cap G_l$,
and $\beta\sim \eps |E'|^{d/(d+1)}|F|^{d/(d+1)}|F|^{-1}$;
we have inserted the relevant values $p=q=\frac{d+1}d$ of the exponents.
Since $T(\chi_E)\ge c\eps|E|^{d/(d+1)}
|F|^{d/(d+1)}|F|^{-1}$
at each point of $G_k\supset G$,
there is the lower bound 
\begin{equation*}
\scriptt(E,G)\gtrsim \eps |E|^{d/(d+1)}
|F|^{d/(d+1)}|F|^{-1}|G|. 
\end{equation*}
The lemma thus yields
\begin{multline*}
|E'|\gtrsim
\Big( \eps |E|^{d/(d+1)} |F|^{d/(d+1)}|F|^{-1}|G||E|^{-1}\Big)^{1/(d-1)}
\\
\cdot \Big( \eps|E'|^{d/(d+1)} |F|^{d/(d+1)} |F|^{-1}  \Big)^{d/(d-1)}.
\end{multline*}
Since $|G|\gtrsim \eps^{2d/(d+1)}|F|$, this implies that
\begin{equation*}
|E'|^{d-1} \gtrsim
\Big( \eps |E|^{-1/(d+1)} |F|^{d/(d+1)} \eps^{2d/(d+1)}  \Big)
\Big( \eps |E'|^{d/(d+1)} |F|^{-1/(d+1)}  \Big)^d.
\end{equation*}
This is equivalent, via a bit of algebra, to
\begin{equation}\label{viaabitofalgebra}
|E'|\le C\eps^{-B}|E|
\end{equation}
for a certain positive exponent $B$.
Since $|E|=|E_k|\sim \eta 2^{-kp}$
and $|E'|=|E_l|\sim \eta 2^{-lp}$,
this last inequality is equivalent to
$2^{-lp}\le C\eps^{-B}2^{-kp}$,
whence  $l\ge k-C'\log(\eps^{-1})$
for a certain finite constant $C'$.
The situation is symmetric in the indices $k,l$,
so the reversed bound also holds.
This contradicts the assumption that $|k-l|\ge A\log(\eps^{-1})$,
provided that the constant $A$ is  chosen to be sufficiently
large at the beginning of the proof.
\end{proof}

Let $q',p'$ be
the exponents conjugate to $q,p$. Then by H\"older's inequality,
\begin{multline*}
\sum_k 2^k\scriptt(E_k,F)
\sim \sum_k 2^k\scriptt(E_k,G_k)
\\
\lesssim \big( \sum_k 2^{kq'}|E_k|^{q'/p} \big)^{1/q'}
(\sum_k|G_k|)^{1/q}
\\
\lesssim \max_k(2^{kp}|E_k|)^\gamma |F|^{1/q}
\lesssim \eta^\gamma |F|^{1/q}
\end{multline*}
for a certain exponent $\gamma$ which is strictly $>0$,
because $\frac1p+\frac1q>1$. We've invoked the
normalization $\sum_k 2^{kp}|E_k|=1$.

An alternative bound is also available.
The number $M$ of indices $k$ in the sum satisfies 
$M\sim \eta^{-1}$, so
\begin{equation}
\begin{split}
\sum_k 2^k\scriptt(E_k,F)
&\sim \sum_k 2^k\eps|E_k|^{1/p}|F|^{1/q}
\\
&\lesssim \eps M\eta^{1/p}|F|^{1/q}
= \eps \eta^{-r}|F|^{1/q}
\end{split}
\end{equation}
where $r = 1-p^{-1}$ is positive.

If the restriction that $|k-l|\ge A\log(1/\eps)$ for distinct
indices $k,l$ is now dropped,
but the normalizations involving $\eta,\eps$ are retained,
then we conclude that 
$\langle Tf,\chi_F\rangle
\lesssim \log(1/\eps)\min(\eta^\gamma, \eps \eta^{-r})|F|^{1/q}$
for certain positive, finite exponents $\gamma,r$.
Therefore
\begin{equation}
\langle Tf,\chi_F\rangle
\lesssim \min(\eps^a,\eta^b)\|f\|_{L^p}|F|^{1/q}
\end{equation}
for certain positive exponents $a,b$,
for all $f,F$ subject to the normalizations involving $\eps,\eta$.
This in turn implies that
\begin{equation} \label{penultimate}
\langle Tf,\chi_F\rangle
\le C\eps^a\|f\|_{L^p}|F|^{1/q}
\end{equation}
for all $f,F$, subject only to the normalization involving $\eps$.
Summing one more series yields the weak type bound
$C\|f\|_{L^p}|F|^{1/q}$ for arbitrary $f,F$;
but \eqref{penultimate} will be used below.

It is now a simple matter to repeat this argument to pass from
the weak type $(p,q')$ inequality to the corresponding strong type
inequality.
Let $g=\sum_j 2^j\chi_{F_j}$,
let $f = \sum_k 2^k \chi_{E_k}$,
and assume that $\|f\|_{L^p}=\|g\|_{L^q}=1$.
Let $\eps,\eta \in(0,\tfrac12]$.
Suppose that $|E_k|\sim\eta 2^{-kp}$ for all indices $k$
for which $|E_k|>0$; drop all other indices $k$.
Consider $\sum_{j,k}^* 2^j2^k\scriptt(E_k,F_j)$,
where a $*$ indicates that a sum is taken only
$j$, $k$, or pairs $(j,k)$ such that 
$\scriptt(E_k,F_j)\sim \eps |E_k|^{1/p}|F_j|^{1/q}$.
At the expense of a factor $\lesssim\log(\eps^{-1})$
we may assume that $|k_1-k_2|\ge A\log(\eps^{-1})$ for
all distinct indices $k_1,k_2$ in the sum representing $f$.

Just as above, to each pair $(j,k)$ is associated a set $G_{j,k}\subset F_j$,
such that $\scriptt(E_k,F_j)\sim \scriptt(E_k,G_{j,k})$
and $\sum^*_k |G_{j,k}|\lesssim |F_j|$.
Then
\begin{align*}
\sum^*_{j,k}2^j2^k\scriptt(E_k,F_j)
&\lesssim 
\sum^*_{j,k}2^j2^k\scriptt(E_k,G_{j,k})
\\
&= \sum_k 2^k \langle T( \chi_{E_k}), \sum^*_j 2^j \chi_{G_{j,k}} \rangle
\\
&\lesssim 
\sum_k 2^k |E_k|^{1/p} 
(\sum^*_j 2^{jq}|G_{j,k}|)^{1/q}.
\end{align*}
To obtain the last line we have invoked the weak type inequality
established above, for the transpose of $T$, which is the same as $T$. 
By H\"older's inequality 
and the bound $\sum^*_k |G_{j,k}|\lesssim|F_j|$
this last line is
\begin{equation}
\label{lorentzimprovementkey}
\begin{split}
\lesssim 
(\sum_k 2^{kq'}|E_k|^{q'/p})^{1/q'} 
&(\sum_k \sum^*_j 2^{jq}|G_{j,k}|)^{1/q}
\\
&\lesssim
\eta^\gamma 
( \sum_j 2^{jq}|F_{j}|)^{1/q}
\lesssim\eta^\gamma.
\end{split}
\end{equation}

On the other hand, if $M$ is the number of indices $k$ then
by applying \eqref{penultimate} to the transpose operator we conclude that
\begin{equation}
\begin{split}
\sum_{j,k}^* 2^j 2^k \scriptt(E_k,F_j)
&\lesssim \eps^a 
\sum_k 2^k|E_k|^{1/p}(\sum_j 2^{jq}|F_j|)^{1/q}
\\
&\lesssim M\eps^a \eta^{1/p}
= \eps^a\eta^{-r}.
\end{split}
\end{equation}
As in the proof of the weak-type bound, summation over dyadic values
of $\eps$ and $\eta$ leads to the desired strong type inequality.
\qed

\begin{proof}[Proof of Theorem~\ref{thm:Lorentz}] 
The Lorentz space bound 
is implicit in the above argument. The dual of $L^{d+1,r}$
is $L^{(d+1)/d, r'}$ where $r'=r/(r-1)$. Thus in the first
factor of the first line of \eqref{lorentzimprovementkey},
one has control over $\sum_k 2^{kr'}|E_k|^{r'/p}$.
A positive power of $\eta$ is therefore obtained
in the second line of \eqref{lorentzimprovementkey}
provided that $q'>r'$. Here $q=(d+1)/d$,
so $q'>r'$ is equivalent to $r>(d+1)/d$.
The only other difference 
is that $M$ is now majorized by a different power of $\eta$,
but all that is needed in the argument is some negative power.
\end{proof}

For the characterization of quasiextremals, we need the following
more quantitative form of the strong type inequality,
which was implicitly established in the course of the proof.
\begin{lemma} \label{lemma:strongtypewithgain}
There exist $\gamma>0$ and $C<\infty$ with the following property.
Let $f=\sum_{k\in\integers} 2^k\chi_{E_k}$
and $f^\star=\sum_{l\in\integers} 2^l\chi_{F_l}$,
where $\{E_k\}$ are pairwise disjoint, and likewise $\{F_l\}$
are pairwise disjoint.
If $2^l|F_l|^{d/(d+1)} \le\eta\|f^\star\|_{L^{(d+1)/d}}$
for all $l$ then 
\begin{equation} 
\langle Tf,f^\star\rangle \le C\eta^\gamma 
\|f\|_{L^{(d+1)/d}}
\|f^\star\|_{L^{(d+1)/d}}.
\end{equation}
\end{lemma}

We also digress to record the following lemma, whose proof is
implicit in the above derivation of \eqref{viaabitofalgebra}.
\begin{lemma} \label{lemma:comparablemeasures}
For any $d\ge 2$ there exist $C,C'<\infty$ with the following
property.
Let $E,E',F\subset\reals^d$ be measurable sets with 
positive, finite measures.
Let $\eta>0$.
If 
$T\chi_E(x)\ge \eta |E|^{d/(d+1)}|F|^{-1/(d+1)}$
and
$T\chi_{E'}(x)\ge \eta |E'|^{d/(d+1)}|F|^{-1/(d+1)}$
for every $x\in F$,
then $|E'|\le C\eta^{-C}|E|$.
\end{lemma}

\section{Quasiextremals for the strong type inequality}

\begin{proof}[Proof of Theorem~\ref{thm:functions}]
Let $f,f^\star$ be any nonnegative measurable functions which
are finite almost everywhere. There exist measurable sets $E_k,F_l$
as in Lemma~\ref{lemma:strongtypewithgain}
such that 
$\tfrac12 f\le \sum_{k\in\integers} 2^k\chi_{E_k}\le f$
and
$\tfrac12 f^\star\le \sum_{l\in\integers} 2^l\chi_{E_l}\le f^\star$.

Unless (with the above notation)
$\sup_l 2^l|F_l|^{d/(d+1}\gtrsim \eps^C\|f^\star\|_{(d+1)/d}$,
Lemma~\ref{lemma:strongtypewithgain} implies that 
$|\langle Tf^\star,f\rangle| \ll\eps\|f\|_{(d+1)/d}\|f^\star\|_{(d+1)/d}$, 
contradicting the hypothesis that $(f,f^\star)$ is $\eps$-quasiextremal.
In the same way it follows that $\sup_k 2^k|E_k|^{d/(d+1)}
\gtrsim \eps^C\|f\|_{(d+1)/d}$.
All sets $E_k,F_l$ not satisfying these inequalities can be discarded.
If none of the remaining pairs $(E_k,F_l)$ were $c\eps^C$-quasiextremal,
then the above reasoning would again imply
$|\langle Tf^\star,f\rangle| \ll\eps\|f\|_{(d+1)/d}\|f^\star\|_{(d+1)/d}$, 
a contradiction.
\end{proof}

This line of argument, leading from a restricted weak type inequality
to a strong type inequality, is rather general. 
See \cite{betsy} for a related application.

\section{Sketch of proof of Theorem~\ref{thm:single}}
In part (i), the first conclusion is a weakening of Theorem~\ref{maintheorem}.
On the other hand,
if $E\subset B^\star = \pi^\star(\scriptb)$
where $\scriptb=\scriptb(\bar z,\earrow,r,r^*)$
then $B=\pi(\scriptb(\bar z,\earrow, Cr,Cr^*))$ satisfies
$|B|\sim|\pi(\scriptb)|$ and
$T^*(\chi_B) \ge c\prod_{j=1}^{d-1} r_j$ at every point of $B^\star\supset E$,
provided that the constants $C$ and $c$ are chosen to be sufficiently
large and small respectively, but independent of $r$.
The stated converse follows from a simple calculation using
the relations $r_jr_j^\star=\rho$ and the definitions of $B,B^\star$.

From the Lorentz space inequality of Theorem~\ref{thm:Lorentz}
and interpolation it follows that $T$ maps $L^{(d+1)/d, \delta+(d+1)/d}$
to $L^{d+1,d+1-\delta}$ for some $\delta>0$.
It follows easily that if $f$ is decomposed as $\sum_{j} 2^j f_j$
where the summands have disjoint supports $E_j$
and satisfy $\chi_{E_j} \le f_j\le 2\chi_{E_j}$ for all $j$,
then there exists $J$ such that $\norm{2^J f_J}_{(d+1)/d}
\ge c\eps^C\norm{f}_{(d+1)/d}$,
and $E_J$ is $c\eps^C$--quasiextremal for the restricted weak type inequality.
Part (i) then gives the stated conclusion for $r=2^J$
and $E=E_J$.

To prove (iii), let $J$ be as in the preceding paragraph and $\scriptb$
be as in the conclusion (ii), and decompose $f = 2^J f_J\cdot\chi_{B^\star} + h$,
where $B^\star = \pi^\star(\scriptb)$.
Since the two summands $h, f_J\chi_{B^\star}$ have disjoint supports,
$\norm{h}_{(d+1)/d}\le (1-c\eps^C)\norm{f}_{(d+1)/d}$.
Let $\Psi = \Psi_{\scriptb}$ be as in the definition of an $\eps$--bump function
associated to $\scriptb$, and consider $F = f_J\chi_{B^\star}\circ\Psi^{-1}$.
Then $F$ is supported on $Q_0$, $\norm{F}_{L^\infty}\le 2$,
and the support of $F$ has measure $\ge c\eps^C$.
Split $F$ as $F = F_{\text{high}}+F_{\text{low}}$ 
into a high-frequency and a low-frequency component,
with the cutoff around frequencies of order of magnitude $\eps^{-A}$.
Using the fact that $T$ is smoothing of positive order in the scale of $L^2$ Sobolev
spaces, it follows readily that if $A$ is chosen to be sufficiently
large, independent of $\eps$, then 
$\norm{T(F_{\text{high}})}_{d+1}$ is small relative to $\norm{T(F)}_{d+1}$.
For the $L^\infty$ norm and support control on $F$ imply similar
control on $F_{\text{high}}$, whence follows an $L^\infty$
bound for 
$T(F_{\text{high}})$ 
which is uniform in $A\ge 1$;
the smoothing property implies an $\lt$ bound for
$T(F_{\text{high}})$ 
which tends to zero as $A\to\infty$;
so interpolation yields a favorable $L^{d+1}$ bound for large $A$.
Multiplying $F_{\text{low}}$ by a suitable spatial cutoff function 
supported in $\pi^\star(\scriptb(\bar z,\earrow, C\eps^{-C}r, C\eps^{-C}r^\star))$
yields a $c\eps^C$--bump function, up to a uniformly bounded constant factor,
with a further remainder term which is again negligible.

Details are left to the dedicated reader.

\section{Verification of Proposition~\ref{prop:trulyQE}}
\label{section:trulyQE}

Let $z=(\bar x,\bar y)\in \scripti$, let $\rho>0$, let $r_jr_j^\star=\rho$
for $j\in\{1,2,\cdots,d-1\}$,
and let $\earrow$ be an orthonormal basis for $\reals^{d-1}$;
all of these parameters are otherwise arbitrary.
We claim that
$\scriptb=\scriptb(\bar z,\earrow,r,r^{\star})$ 
and its projections satisfy
\begin{gather}
|\scriptb|
\gtrsim \rho^{d},
\qquad
|\pi(\scriptb)| \lesssim \rho\prod_{j=1}^{d-1}r_j,
\qquad
|\pi^{\star}(\scriptb)| \lesssim \rho\prod_{j=1}^{d-1}r^{\star}_j
\\
\intertext{whence}
\frac{|\scriptb|}{
|\pi(\scriptb)|^{d/(d+1)}
|\pi^{\star}(\scriptb)|^{d/(d+1)}
}
\gtrsim 1
\end{gather}
uniformly in all these parameters;
thus $(\pi(\scriptb),\pi^\star(\scriptb))$
is a $c_0$-quasiextremal for some constant $c_0$ independent of all parameters.

\begin{proof}
The upper bounds $|\pi(\scriptb)|,|\pi^\star(\scriptb)|$
follow directly from the definition of $\scriptb$,
which is defined to be the intersection of $\scripti$
with a certain Cartesian product $E\times E^\star$.
What must be verified is the lower bound for $|\scriptb|$.
 
Fix a small constant $\eps>0$.
Without loss of generality, we may suppose that $\earrow$
is the standard basis for $\reals^{d-1}$, so that
points $(x,y)\in\scriptb$ satisfy
$|x_j-\bar x_j|<r_j$ and 
$|y_j-\bar y_j|<r^\star_j$ 
for all $j\in\{1,2,\cdots,d-1\}$.
Define $E_\eps$ to be the set of all 
$x=(x',x_d)\in\reals^{d-1}\times\reals$ satisfying
\begin{align*}
&|x'_j-\bar x'_j|<\eps r_j
\text{ for all $j\in\{1,2,\cdots,d-1\}$}
\\
& \big| x_d-\bar y_d - |x'-\bar y'|^2 \big| <\eps\rho. 
\end{align*}
Then $|E_\eps|\gtrsim \eps^d \rho\prod_{j=1}^{d-1}r_j$.

We will show that if $\eps$ is chosen to be sufficiently small
but independent of $z,r_j,r_j^\star,\rho,\earrow$, then
for any $x\in E_\eps$, the set of all $y'\in\reals^{d-1}$
for which there exists $y_d\in\reals$ such that $(x,(y',y_d))\in\scriptb$
has measure $\gtrsim \prod_{j=1}^{d-1}r_j^\star$. 
Since the mapping $\scripti\owns (x,y)\mapsto (x,y')\in\reals^d\times\reals^{d-1}$
is a diffeomorphism, this together with the 
lower bound for $|E_\eps|$ and the identities $r_jr_j^\star\equiv\rho$
implies the required lower bound on $|\scriptb|$.

Let $y\in\reals^{d-1}$
satisfy $|y'_j-\bar y'_j|<r_j^\star$ for all $j\le d-1$, and define
$y_d-x_d=-|y'-x'|^2$, so that $(x,y)\in\scripti$.
Then 
\[
(x,y)\in\scriptb
\ \text{ if and only if }\ 
\big| y_d-\bar x_d + |y'-\bar x'|^2 \big|<\rho,
\]
and we aim to show that this last inequality is satisfied.
One has
\begin{align*}
y_d-\bar x_d &+ |y'-\bar x'|^2 
\\
&=
(x_d - |y'-x'|^2) -\bar x_d + |y'-\bar x'|^2 
\\
&=
\big(\bar y_d + |x'-\bar y'|^2 +O(\eps\rho)\big)
- |y'-x'|^2 -\bar x_d + |y'-\bar x'|^2 
\\
&=
\bar y_d -\bar x_d
+ |x'-\bar y'|^2 
- |y'-x'|^2 
+ |y'-\bar x'|^2 
+O(\eps\rho)
\\
&=
-|\bar y'-\bar x'|^2
+ |x'-\bar y'|^2 
- |y'-x'|^2  
+ |y'-\bar x'|^2 
+O(\eps\rho)
\end{align*}
where ``$O(\eps\rho)$'' signifies a quantity whose
absolute value is at most $\eps\rho$; such quantities
are harmless here.
Substitute $x'=\bar x'+\Delta_x$, $y'=\bar y'+\Delta_y$, 
and
$v=\bar y'-\bar x'$.
Then
\begin{align*}
-|\bar y'-\bar x'|^2
&+ |x'-\bar y'|^2 
- |y'-x'|^2  
+ |y'-\bar x'|^2 
\\
&=
-|v|^2
+ |\Delta_x - v|^2 
- |(\Delta_y-\Delta_x) + v|^2  
+ |\Delta_y+v|^2 
\\
&=
2\langle\Delta_x,\Delta_y\rangle;
\end{align*}
all other terms cancel in pairs after all four 
quantities are squared.
Since 
$|\langle\Delta_x,\Delta_y\rangle|
\le \sum_{j=1}^{d-1} \eps r_jr^\star_j
= (d-1)\eps\rho$, we conclude that
\[
|y_d-\bar x_d + |y'-\bar x'|^2|
\le (2d-1)\eps\rho.
\]
This is $<\rho$ provided that
$\eps$ is chosen to be sufficiently small.
\end{proof}

\begin{remark}
This conclusion could have been obtained by exploiting
symmetries of the problem to reduce the general case to $\bar x=\bar y$;
this boils down to the same algebraic calculations used above.
For instance, writing $x=(x',x_d)$ and $y=(y',y_d)$,
for any $\Delta\in\reals^{d-1}$,
the mappings
\begin{align*}
&(x',x_d;\ y',y_d)
\mapsto (x'+\Delta,x_d;y'+\Delta,y_d)
\\
&(x',x_d;\ y',y_d)
\mapsto
(x'+\Delta, x_d + 2\langle \Delta, x'\rangle+|\Delta|^2;\ 
y',y_d+2\langle\Delta,y'\rangle)
\end{align*}
are each Cartesian products of two measure-preserving transformations
of $\reals^d$, and preserve the incidence manifold $\scripti$.
These symmetries reduce the general case to the case where $z=(\bar x,\bar y)
= (0,x_d;0,x_d)$.
\end{remark}

\section{On subalgebraic structure} \label{section:conjectures}

Consider the general situation of two (small, open) manifolds $X,X^\star$
and a smooth incidence manifold $\scripti\subset X\times X^\star$,
equipped with a nonnegative measure $\sigma$ with a smooth, nonvanishing density.
Assume that the projections $\pi,\pi^\star$ of $\scripti$ onto $X,X^\star$ 
are submersions, and that the two foliations of $\scripti$ defined
by the level sets of $\pi,\pi^\star$ are everywhere transverse.  
Associated to these data is 
$\scriptt(E,E^\star)
=\scriptt_\scripti(E,E^\star)=\sigma({\scripti\cap(E\times E^\star)})$,
the continuum number of incidences between $E$ and $E^\star$.
Assume that there exist some exponents $a,a_\star\in(0,1)$ 
satisfying $a+a_\star>1$ for which there is an $L^p$-improvement inequality 
$\scriptt(E,E^\star)\le C|E|^a|E^\star|^{a_\star}$
uniformly for all measurable sets.
For all $t,t_\star>0$ define
\begin{equation}
\Lambda(t,t_\star) =
\sup_{|E|=t, |E^\star|=t_\star} \scriptt(E,E^\star).
\end{equation}

We say that $\scriptt_\scripti$ {\em has subalgebraic almost-extremals} if 
for every $\delta>0$,
for all sufficiently small positive $t,t_\star$,
there exist sets $E,E^\star$ of measures $t,t_\star$ such that
(i) 
$\scriptt(E,E^\star)\ge c_\delta t^\delta t_\star^\delta \Lambda(t,t_\star)$
and
(ii) $E,E^\star$ are subalgebraic sets of degrees and complexities
bounded above by
quantities depending only on $\delta$, uniformly in $t,t_\star$.
The qualifier ``almost'' refers to the sacrificed factor $t^\delta t_\star^\delta$,
which compensates for an obvious defect: The class of subalgebraic
sets is not compatible with the symmetry group ${\rm Diff(X)}\times{\rm Diff(X^\star)}$
of Cartesian products of diffeomorphisms.

It might seem plausible that
for all $\scriptt_\scripti$ satisfying an $L^p$-improvement inequality,
subalgebraic almost-extremals exist. 
A stronger assertion would be that
any $\eps$-quasiextremal pair has a large subalgebraic subpair.
By this we mean that if $\scriptt(E,E^\star)\ge\eps \Lambda(|E|,|E^\star|)$
then there exist subalgebraic sets $\scripte,\scripte^\star$, 
of uniformly bounded degrees and complexities,
whose measures are comparable to the measures of $E,E^\star$ respectively,
such that 
$\scriptt(E\cap\scripte,E^\star\cap\scripte^\star)\ge c\eps^A|E|^\delta|E^\star|^\delta
\scriptt(E,E^\star)$.
But this stronger assertion is false,
as was shown above in the discussion following
the statement of Theorem~\ref{thm:Lambda}.
That discussion demonstrates it can only
hold for a limited regime of values of $(|E|,|E_\star|)$.
Perhaps a restriction related to the inequality
$\Lambda(|E|,|E^\star)\ll \min(|E|,|E^\star|)$ could be sufficient
to rectify matters in many cases.

It would be desirable to go still further, by describing all
quasiextremals for very general incidence manifolds,
as Theorem~\ref{maintheorem} does for one example.
In certain other contexts, one would like
quasiextremals to correspond to appropriate subalgebraic sets in phase space.

\begin{remark}
It is informative to consider
Young's convolution inequality
\begin{equation} \label{eq:young}
\Big|\iint_{\reals^{2}}f(x)g(y)h(x-y)\,dx\,dy\Big|
\le C\|f\|_p\|g\|_q\|h\|_r, 
\end{equation}
where $p\rp+q\rp+r\rp=2$,
from this perspective, even though
\eqref{eq:young} is not an inequality of precisely the type under consideration
here, partly because it concerns a trilinear rather than a bilinear form,
but primarily because it lacks an appropriate analogue of the
$L^p$-improving property; natural choices of the associated vector fields 
in the incidence manifold form Abelian Lie algebras.
Let $\delta>0$ be small.
Taking $f,g,h$ to be intervals
of some common length $\delta$, 
centered at the origin,
produces subalgebraic quasiextremals. 
But for large $N$, taking each function to be an $N^{-1}\delta$-neighborhood of
$\{N^{-1}n: n\in\integers \text{ and } |n|\le N\}$
produces equally optimal quasiextremals,
uniformly in $N,\delta$ so long as $0<\delta\le\tfrac14$.
While these sets are subalgebraic, their
complexity tends to infinity with $N$,
provided that $\delta$ and $N$ are coupled so that $\delta\to 0$
as $N\to\infty$.
Thus subalgebraic almost-extremals and even quasiextremals exist,
but it is not true that any quasiextremal 
has a large subalgebraic subpair with appropriate complexity bounds. 
Subalgebraic sets of controlled complexityof controlled complexity are not the appropriate class for such Abelian inequalities. 
\end{remark}

Finite lattices also arise as quasiextremals for the Szemer\'edi-Trotter
inequality concerning incidences between discrete sets of lines and points in $\reals^2$. 

\begin{remark}
There is an analogy with a result in discrete 
combinatorics, in which subalgebraic sets are replaced by finite
arithmetic multiprogressions. 
Let $A,B$ be sets of integers of cardinalities comparable to $k$.
Let $S\subset A\times B$ have cardinality comparable
to $k^2$. Suppose that the cardinality of $\{a+b: (a,b)\in S\}$
is comparable to $k$.
Then there exists a subset $A'\subset A$ of cardinality
comparable to $k$, which is contained in a finite arithmetic
multiprogression of uniformly bounded rank, 
whose cardinality is comparable to $k$.
This is a direct consequence of theorems of Balog-Szemer\'edi and Freiman;
see \cite{nathanson}.
\end{remark}

\end{document}